\pdfoutput=1
\RequirePackage{ifpdf}
\ifpdf 
\documentclass[pdftex]{sigma}
\else
\documentclass{sigma}
\fi

\newcommand{\indic}{1\negthickspace\text{I}}

\newcommand{\eps}{\varepsilon}
\newcommand{\fhi}{\varphi}
\newcommand{\dd}{\text{d}}

\DeclareMathOperator{\spec}{Spec}
\DeclareMathOperator{\dha}{d_H}\DeclareMathOperator{\diag}{diag}
\DeclareMathOperator{\Span}{Span}\DeclareMathOperator{\id}{Id}
\DeclareMathOperator{\ch}{ch}\DeclareMathOperator{\cl}{cl}

\numberwithin{equation}{section}

\newtheorem{Theorem}{Theorem}[section]
\newtheorem{Lemma}[Theorem]{Lemma}
\newtheorem{Proposition}[Theorem]{Proposition}
 { \theoremstyle{definition}
\newtheorem{Example}[Theorem]{Example}}

\begin{document}
\allowdisplaybreaks

\newcommand{\arXivNumber}{1803.06537}

\renewcommand{\PaperNumber}{085}

\FirstPageHeading

\ShortArticleName{Renormalization of the Hutchinson Operator}
\ArticleName{Renormalization of the Hutchinson Operator}

\Author{Yann DEMICHEL}
\AuthorNameForHeading{Y.~Demichel}

\Address{Laboratoire MODAL'X - EA3454, Universit\'{e} Paris Nanterre,\\ 200 Avenue de la R\'{e}publique, 92000 Nanterre, France}
\Email{\href{mailto:yann.demichel@parisnanterre.fr}{yann.demichel@parisnanterre.fr}}
\URLaddress{\url{https://www.parisnanterre.fr/m-yann-demichel--699219.kjsp}}

\ArticleDates{Received March 20, 2018, in final form August 10, 2018; Published online August 16, 2018}

\Abstract{One of the easiest and common ways of generating fractal sets in ${\mathbb R}^D$ is as attractors of affine iterated function systems (IFS). The classic theory of IFS's requires that they are made with contractive functions. In this paper, we relax this hypothesis considering a new operator $H_\rho$ obtained by renormalizing the usual Hutchinson operator $H$. Namely, the $H_\rho$-orbit of a given compact set $K_0$ is built from the original sequence $\big(H^n(K_0)\big)_n$ by rescaling each set by its distance from $0$. We state several results for the convergence of these orbits and give a geometrical description of the corresponding limit sets. In particular, it provides a way to construct some eigensets for $H$. Our strategy to tackle the problem is to link these new sequences to some classic ones but it will depend on whether the IFS is strictly linear or not. We illustrate the different results with various detailed examples. Finally, we discuss some possible generalizations.}

\Keywords{Hutchinson operator; iterated function system; attractor; fractal sets}

\Classification{28A80; 37C70; 47H10; 37C25; 37E05; 15A99}

\section{Introduction and notation}\label{sec:intro}

The theory and the use of fractal objects, introduced and developed by Mandelbrot (see, e.g.,~\cite{mand1}), still play an important role today in scientific areas as varied as physics, medicine or finance (see, e.g.,~\cite{falco} and references therein). Exhibit theoretical models or solve practical problems requires to produce various fractal sets. There is a long history of generating fractal sets using Iterated Function Systems. After the fundamental and theo\-re\-ti\-cal works by Hutchinson (see \cite{hutch1}), this method was popularized and developed by Barnsley in the 80s (see~\cite{barn, bardem}). Since these years very numerous developments and extensions were made (see, e.g.,~\cite{barw}) making even more enormous the literature related to these topics. Indeed, the simplicity and the efficiency of this approach have contributed to its success in a lot of domains, notably in image theory (see, e.g.,~\cite{fish}) and shape design (see, e.g.,~\cite{gent}).

\subsection{Background}\label{subsec:background}

Let us recall the mathematical context and give the main notation used throughout the paper. Let $(M,\dd)$ be a metric space. For any map $f\colon M \to M$, we define the $f$-orbit of a point $x_0\in M$ as the sequence $(x_n)_n$ given by
\begin{gather*}\label{defi:orbit}
x_n = (f\circ \cdots \circ f)(x_0) =f^n(x_0),
\end{gather*}
where $f^n$ is the $n$th iterate of $f$ with the convention that~$f^0$ is the identity function~$\id$. In particular, one has $x_{n+1}=f(x_n)$ hence, if $f$ is continuous and if $(x_n)_n$ converges to $z\in M$, then~$z$ is an invariant point for~$f$, i.e., $f(z)=z$.

We denote by ${\cal K}_M$ the set of all non-empty compact subsets of $M$. We obtain a metric space endowing it with the Hausdorff metric $\dha$ defined by
\begin{gather*}\label{def:dh}
\forall\,K,K'\in {\cal K}_M,\qquad \dha(K,K')=\inf\big\{\eps>0 \,|\, K\subset K'(\eps) \text{ and } K'\subset K(\eps)\big\},
\end{gather*}
where $K(\eps)$ is the set of points at a distance from $K$ less than $\eps$.

For every $K\subset M$ we define the set $f(K)=\{f(x) \colon x\in K\}$ and we will assume in the sequel that $f(K)\in {\cal K}_M$ if $K\in{\cal K}_M$.

Let us consider $p\geqslant1$ maps $f_1,\ldots,f_p$ with $f_i\colon M\to M$. Then we can define a new map $H\colon {\cal K}_M \to {\cal K}_M$ setting
\begin{gather*}
\forall\,K\in {\cal K}_M, \qquad H(K)= \bigcup_{i=1}^p f_i(K).
\end{gather*}
We say that $H$ is the \textit{Hutchinson operator} associated with the \textit{iterated function system} (IFS in short) $\{f_1,\ldots,f_p\}$ (see, e.g., \cite{barn,falco, hutch1}).

Basic questions about an IFS are the following: Does the orbit $\big(H^n(K_0)\big)_n$ converge for any compact set $K_0$? Does its limit depend on $K_0$? What are the geometrical properties of the limit sets?

The classic theory of IFS's is based on the contractive mapping principle (see, e.g., \cite{barn,falco, hutch1}). Let us recall that a map $f\colon M\rightarrow M$ is contractive if
\begin{gather*}\label{eq:cont}
\lambda_f = \sup \left\{\frac{\dd(f(x),f(y))}{\dd(x,y)}\colon x,y\in M \text{ with } x\neq y \right\} < 1.
\end{gather*}

Let us assume that $(M,\dd)$ is a complete metric space. Then, any contractive map is continuous, has a unique invariant point $z\in M$, and the $f$-orbit of any $x_0\in M$ converges to $z$ with the basic estimate
\begin{gather*}\label{eq:estimrate}
\forall\,n\geqslant0,\qquad \dd\big(f^n(x_0),z\big)\leqslant \lambda_f^n \dd(x_0,z).
\end{gather*}

If $f_1,\ldots,f_p$ are contractive then the associated Hutchinson operator $H$ is also contractive because of
\begin{gather*}\label{eq:contH}
\lambda_H= \max_{1\leqslant i \leqslant p}\{\lambda_{f_i}\}.
\end{gather*}
Since $({\cal K}_M,\dha)$ inherites the completeness of $(M,\dd)$, the map $H$ has then a unique invariant point $L\in {\cal K}_M$, called the \textit{attractor} of $H$, and for all $K_0\in {\cal K}_M$ the sequence $\big(H^n(K_0)\big)_n$ converges to $L$. One of the interests is that such sets $L$ are generally fractal sets.

In the sequel, the space $M$ will be essentially ${\mathbb R}^D$, $D\geqslant1$, endowed with the metric induced by the Euclidean norm $\|\cdot\|$. Writing simply ${\cal K}$ for ${\cal K}_M$, a subset $K\subset {\mathbb R}^D$ belongs to~${\cal K}$ if and only if it is closed and bounded. In particular, the closed ball with center $x\in {\mathbb R}^D$ and radius $r>0$ will be denoted by $B(x,r)$.

In this paper, we are interested in affine IFS's, i.e., when $f_i$ is defined by $f_i(x)= A_ix + b_i$ with $A_i$ a $D\times D$ matrix and $b_i\in{\mathbb R}^D$ a vector. Such a map satisfies $\lambda_{f_i}=\|A_i\|$ where $\|A_i\|$ is the norm of $A_i$ given by
\begin{gather*}\label{def:normmatrix}
\|A_i\| = \sup\big\{\|A_ix\| \colon x\in {\mathbb R}^D \text{ with } \|x\|=1\big\} = \inf\big\{r>0\,|\,\forall\,x\in {\mathbb R}^D,\, \|A_ix\|\leqslant r\|x\|\big\}.
\end{gather*}
In particular, classic IFS's consist of transformations involving rotations, symmetries, scalings and translations. In this case, if $H$ is contractive, the corresponding attractor $L$ is called a~self-affine set.
One obtains a nice subclass of such IFS's when the $f_i$'s are homotheties, i.e., when $f_i(x)=\alpha_i x + b_i$ with $\alpha_i\geqslant 0$. Indeed, contrarily to general affine maps, $f_i$ contracts the distances with the same ratio $\alpha_i$ in all directions. This enables a precise description of $L$. For example, if the sets $f_i(K_0)$ are mutually disjoints then $L$ is a Cantor set whose fractal dimension is the solution of a very simple equation (see \cite{falco,perso}). Cantor sets are fundamental and come naturally when one studies IFS's. A simple family of Cantor sets in ${\mathbb R}$ is $\big\{\Gamma_a \colon 0<a<\frac1{2}\big\}$ where $\Gamma_a$ is the attractor of the IFS $\{f_1,f_2\}$ with $f_1(x)=ax$ and $f_2(x)=ax+(1-a)$. For example, $\Gamma_{\frac1{3}}$ is the usual triadic Cantor set (see~\cite{edgar,falco, hutch1}). When $\frac1{2}\leqslant a<1$, the attractor of the previous IFS becomes the whole interval $[0,1]$. These basic examples will be extensively used in the sequel.

\subsection{Motivation}\label{subsec:motivation}

Let us point out two specific situations:
\begin{enumerate}\itemsep=0pt
\item[--] When $\lambda_H\geqslant 1$ the previous results become false: typical orbits fail to converge. Basically, the orbits of some points $x_0\in K_0$ may then satisfy $\|f_i^n(x_0)\|\to\infty$ for some $i$, preventing the sequence $\big(H^n(K_0)\big)_n$ from being bounded.
\item[--] When all the $f_i$'s are contractive linear maps, the attractor of $H$ is always $\{0\}$ so does not depend on the fine structure of the $A_i$'s but only on their norms.
\end{enumerate}

However, in these two degenerate situations we can observe an intriguing geometric structure of the sets $H^n(K_0)$. For example, let us consider the IFS $\{f_1,f_2\}$ where the $f_i\colon {\mathbb R}^2\to {\mathbb R}^2$ are the linear maps given by their canonical matrices
\begin{gather*} A_1 = \left[\begin{matrix}
a & a \\
a & 0
\end{matrix}\right] \qquad \text{and}\qquad
A_2 = \left[ \begin{matrix}
a & -a \\
-a & 0
\end{matrix}\right]\end{gather*}
with $a>0$. We focus on the $H$-orbit of the unit ball $B(0,1)$. For all $a$ large enough we have $\|A_1\|=\|A_2\|>1$ and the sequence $\big(H^n(B(0,1))\big)_n$ is not bounded: the diameter $d_n$ of $H^n(B(0,1))$ grows to infinity. At the contrary, for all $a$ small enough we have $\|A_1\|=\|A_2\|<1$. Thus $H$ is now contractive and $\big(H^n(B(0,1))\big)_n$ converges to $\{0\}$: $d_n$ vanishes to $0$. Nevertheless, whatever is $a$, one can observe that the sets $H^n(B(0,1))$ tend to a same limit shape looking like a `sea urchin'-shaped set (see Fig.~\ref{fig:oursin}). So one can wonder if there exists a critical value $\overline{a}$ for which $d_n$ do not degenerate so makes possible to observe this asymptotic set.
\begin{figure}[!h]\centering
\begin{minipage}{52mm}\centering
\includegraphics[scale=0.20]{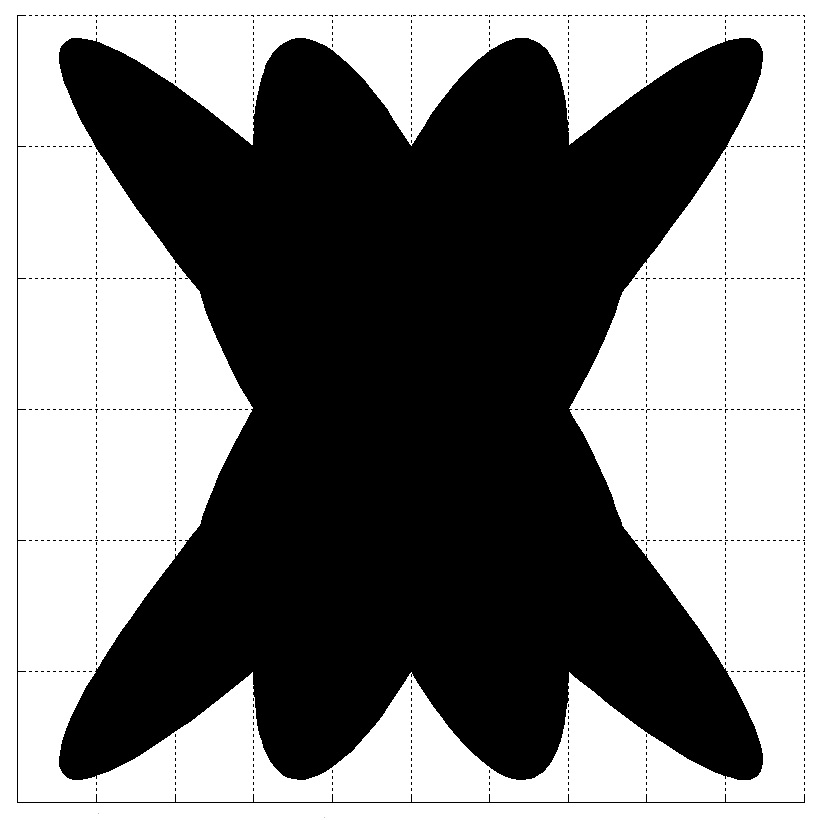}\\
{\footnotesize (a) $H^2(B(0,1))$}
\end{minipage}\
\begin{minipage}{52mm}\centering
\includegraphics[scale=0.20]{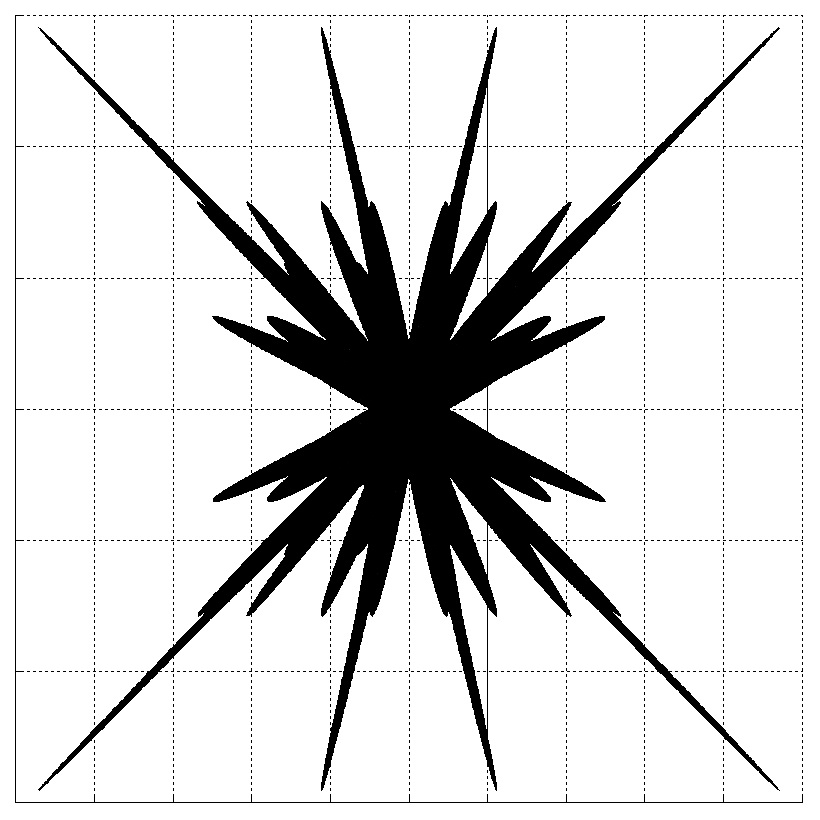}\\
{\footnotesize (b) $H^5(B(0,1))$}
\end{minipage}\
\begin{minipage}{52mm}\centering
\includegraphics[scale=0.20]{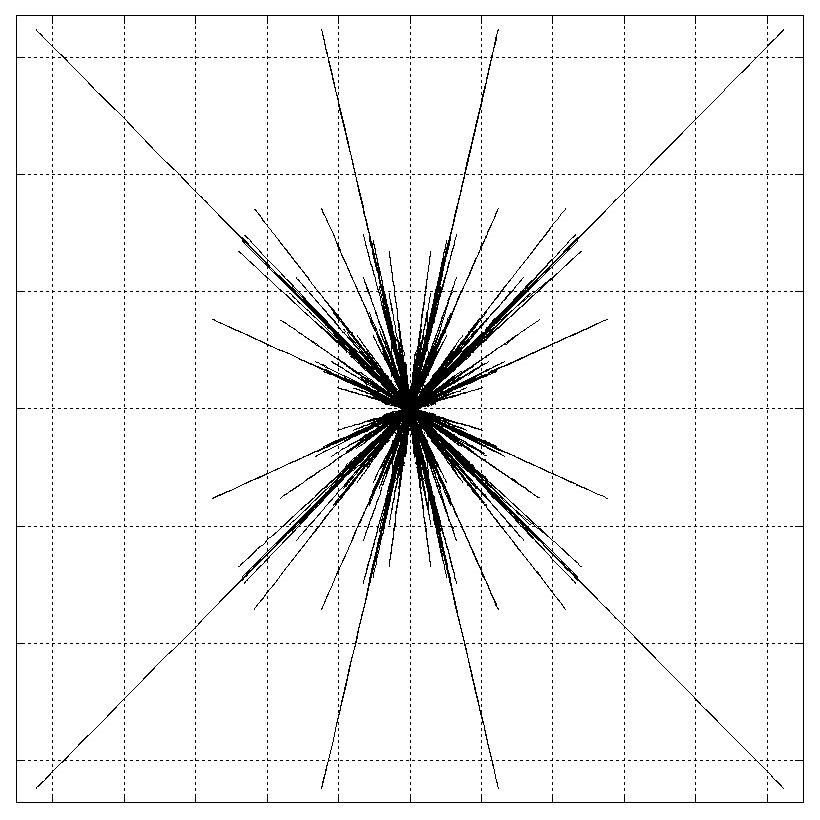}\\
{\footnotesize (c) $H^{10}(B(0,1))$}
\end{minipage}
\caption{Three sets of the $H$-orbit of $B(0,1)$ where $H$ is the Hutchinson operator associated with the IFS $\{f_1,f_2\}$. Maps $f_1$, $f_2$ are given above. Since the maps are linear, changing the parameter~$a$ gives the same sets for each~$n$ up to a scaling factor. Thus an adequate renormalization should reveal a common asymptotic limit-shape.}\label{fig:oursin}
\end{figure}

In this paper, we aim to modify the original Hutchinson operator to annihilate these two degenerate behaviors. We wish to obtain a limit set even if the IFS is not contractive, and a non zero limit set for contractive linear IFS's. Moreover, we would like this new operator to exhibit the typical `limit shape' observed above.

\subsection{Renormalization with the radius function}\label{subsec:strategy}

Our strategy is to rescale each set $H^n(K_0)$ by dividing it by its size. The idea of rescale a~sequence of sets $(K_n)_n$ to get its convergence to a non degenerate compact limit is not new and is particularly used in stochastic modeling (see, e.g., \cite{cox81, rich73} for famous examples of random growth models and more recently \cite{legall07, marck06} in the context of random graphs and planar maps). Probabilists usually consider the \textit{a posteriori} rescaled sets $\frac1{d_n}K_n$ where $d_n$ estimates the size of~$K_n$, often its diameter.

Here we proceed differently. First, in order to keep dealing with the orbit of an operator, we will do an \emph{a priori} renormalization. Secondly, we will measure the size of a compact set with its distance from $0$. Precisely, we consider the radius function $\rho$ defined on ${\cal K}$ by
\begin{gather*}
\forall\,K\in{\cal K},\qquad\rho(K)=\sup\{\|x\| \colon x\in K\}
\end{gather*}
and we denote by $H_{\rho}$ the operator defined by
\begin{gather}\label{def:normhutch}
\forall\,K\in{\cal K},\qquad H_{\rho}(K)= \frac1{\rho(H(K))} H(K).
\end{gather}

The radius function $\rho$ satisfies the three following basic properties:
\begin{enumerate} \itemsep=0pt
\item[--] continuity: $\rho$ is continuous with respect to $\dha$;
\item[--] monotonicity: If $K\subset K'$ then $\rho(K)\leqslant \rho(K')$;
\item[--] homogeneity: For all $\alpha\in{\mathbb R}$, $\rho(\alpha K)= |\alpha|\rho(K)$.
\end{enumerate}

Actually $\rho$ is a very nice function because it enjoys an additional stability property:
\begin{gather}\label{eq:rhostab}
\forall\,K,K'\in {\cal K}, \qquad \rho(K\cup K') = \max\{\rho(K),\rho(K')\}.
\end{gather}

The subject of interest of the paper is then the $H_{\rho}$-orbit of sets $K_0\in{\cal K}$. For simplicity, we will write in the sequel $K_n = H_{\rho}^n(K_0)$ so that
\begin{gather}\label{def:normhutchiter}
\forall\,n\geqslant 0,\qquad K_{n+1}= \frac1{d_n}\bigcup_{i=1}^p f_i(K_n)
\end{gather}
with
\begin{gather}\label{eq:defdn}
d_n=\rho\left(\bigcup_{i=1}^p f_i(K_n)\right)= \max_{1\leqslant i \leqslant p} \rho(f_i(K_n)).
\end{gather}

We will assume that $d_n>0$, i.e., $K_n\neq \{0\}$.

Observe that $\rho(K_n)=1$ for all $n\geqslant1$, thus:
\begin{enumerate}\itemsep=0pt
\item[--] $K_n\subset B(0,1)$ so that the orbit of any set $K_0$ is bounded;
\item[--] there exist at least one $x_n\in K_n$ such that $\|x_n\|=1$ so that $(K_n)_n$ cannot vanish to $\{0\}$.
\end{enumerate}
In particular, if $(K_n)_n$ converges to a set $K$ then $\rho(K)=1$ and $K\neq\{0\}$.

This new operator $H_{\rho}$ is then a good candidate to solve the problems discussed in Section \ref{subsec:motivation}. It will act by freezing the geometrical structure of $H^n(K_0)$ at each step $n$ of the construction of the orbit.

\subsection{Eigen-equation problem}\label{subsec:eigenequation}

Let us point out a very strong connection with the `eigen-equation problem' recently studied in~\cite{barvin11} for affine IFS's. Indeed, if $(K_n)_n$ converges to a set $K$ then $(d_n)_n$ converges to $d>0$ and taking the limit in \eqref{def:normhutchiter} leads to $H(K)=d K$. Hence $d$ is an eigenvalue of $H$ and $K$ a corresponding eigenset. Existence of solutions for this equation is discussed and proved in~\cite{barvin11}. The values for~$d$ are closely related to the joint spectral radius~$\sigma_{{\cal M}}$ of the~$A_i$'s (see~\eqref{eq:JSR}). In particular, for linear IFS's, $\sigma_{{\cal M}}$ was interpreted as a transition value for which exists a~corresponding eigenset~$K$ whose structure is similar to the one described in Section~\ref{subsec:motivation}. Unfortunately, these results don't hold for every IFS. In particular it rules out simple IFS's only made up with homotheties or some more interesting ones made up with stochastic matrices. However, the results stated in~\cite{barvin11} provide important clues to determine and study the possible limits of both sequences~$(d_n)_n$ and~$(K_n)_n$.

When studying the eigen-equation problem, an interesting question is to approximate any couple $(d,K)$ of solutions of equation $H(K)=dK$. Let us look at the special case when the IFS consists in only one linear map with matrix $A$ and set $K_0=\{x_0\}$. Then $K_n=\{x_n\}$ with
\begin{gather*}\label{def:poweriter}
\forall\,n\geqslant 0,\qquad x_{n+1}= \frac1{\|A x_n\|}\,A x_n.
\end{gather*}
One recognizes the famous \textit{power iteration algorithm}. With suitable assumptions it gives a~simple way to approximate the unit eigenvector associated with the dominant eigenvalue $\sigma_{{\cal M}}$ of $A$, this eigenvalue being the limit of $d_n= \|A x_n\|$. Therefore, iterating the operator $H_{\rho}$ from a~set~$K_0$ is nothing but a generalization of this algorithm and then provides a natural procedure to approximate both an eigenvalue of~$H$ and one of its associated eigenset.

From now on we are then interested in the convergence of $(K_n)_n$ and the geometric properties of its limit. Typically, $H_{\rho}$ is not contractive and the classic theory may not be applied. In particular, $H_{\rho}$ may have different invariant points so that the limit of $(K_n)_n$ may be no longer unique but deeply depend on $K_0$. Furthermore, it is clear that the $H_\rho$-orbits of $K_0$ may diverge for some $K_0$ (for example when the $A_i$'s are only rotations). We will expose different ways to state the convergence of $(K_n)_n$ depending on whether the IFS is affine (Section~\ref{sec:affine}) or strictly linear (Section~\ref{sec:linear}). Finally, some generalizations will be shown in the last section (Section~\ref{sec:general}).

\section{Results for affine IFS's}\label{sec:affine}

We suppose in this section that the IFS consists in $p\geqslant1$ affine maps $f_i\colon {\mathbb R}^D\rightarrow {\mathbb R}^D$ defined by $f_i(x)=A_i x+b_i$. We denote by ${\cal M}=\{A_1,\ldots,A_p\}$ the set of their canonical matrices. Let us recall that the joint spectral radius of ${\cal M}$ is defined by
\begin{gather}\label{eq:JSR}
\sigma_{{\cal M}} = \limsup_{n\to\infty} \big(\sup_{1\leqslant i_1,\ldots,i_n\leqslant p}\{\alpha(A_{i_1}\cdots A_{i_n})\}\big)^{\frac1 n}
= \limsup_{n\to\infty} \big(\sup_{1\leqslant i_1,\ldots,i_n\leqslant p}\{\|A_{i_1}\cdots A_{i_n}\|\}\big)^{\frac1 n},\!\!\!
\end{gather}
where $\alpha(M)$ denotes the usual spectral radius of the matrix $M$ (see \cite{theys}). Finally, we denote by $\spec(M)$ the set of the eigenvalues of $M$ so that $\alpha(M)=\max\{ |\alpha| \colon \alpha\in\spec(M)\}$.

\subsection{Strategy: a general result}\label{sec:mainsize}

Our strategy consists in linking the convergence of $(K_n)_n$ to the asymptotic behavior of the sequence of positive numbers $(d_n)_n$. If $(K_n)_n$ converges to a set $K$ then $(d_n)_n$ converges to $d>0$ and the eigen-equation $H(K)=dK$ shows that $K$ may be seen as an invariant set of the classical Hutchinson operator $H_d=\frac1{d}H$ associated with the IFS $\{\frac1{d}f_1,\ldots,\frac1{d}f_p\}$. In particular, if $d>\lambda_H$ then $K$ is unique: it is the attractor $L_d$ of this contractive operator $H_d$.

Conversely, if $(d_n)_n$ is a constant sequence, say $d_n=d$, one has $K_n=H_d^n(K_0)$ so that $(K_n)_n$ converges to $L_d$ if $d>\lambda_H$. Actually, when $(d_n)_n$ is no longer constant, but converges to a~positive number $d>\lambda_H$, then the convergence of $(K_n)_n$ to $L_d$ still happen.

\begin{Theorem}\label{theo:renorm}
Let $K_0\in{\cal K}$. Assume that the sequence $(d_n)_n$ converges to a positive number $d>\lambda_H$. Then the sequence $(K_n)_n$ converges to the attractor $L_d$ of the $\mathrm{IFS}$ $\big\{\frac1{d}f_1,\ldots,\frac1{d}f_p\big\}$.
\end{Theorem}

\begin{proof} Let us set $K'_n=H_d^n(K_0)$. We have to prove that $\eps_n=\dha(K_n,K'_n)$ converges to $0$. We can write
\begin{align*}
\dha(K_{n+1},K'_{n+1}) & \leqslant \dha\left(\frac1{d_n}H(K_n),\frac1{d_n}H(K'_n)\right)+ \dha\left(\frac1{d_n}H(K'_n),\frac1{d}H(K'_n)\right) \\
& \leqslant \frac{\lambda_H}{d_n}\dha(K_n,K'_n)+\left|\frac1{d_n}-\frac1{d}\right|\rho(H(K'_n)).
\end{align*}

Since $(K'_n)_n$ converges, there exists $B\in{\cal K}$ such that $K'_n\subset B$ for all $n\geqslant0$. Then let us fix $\eta>0$ and $N\geqslant 0$ such that $0<\lambda_H<d-\eta\leqslant d_n \leqslant d+\eta$ for all $n\geqslant N$. We obtain $0\leqslant \eps_{n+1} \leqslant \mu \eps_n + m_n$ where $\mu=\frac{\lambda_H}{d-\eta}$ and $m_n=\big|\frac1{d_n}-\frac1{d}\big|\rho(H(B))$.
It follows that
\begin{gather*}
\forall\,n>N,\qquad 0\leqslant \eps_n\leqslant \mu^{n-N} \eps_N + \sum_{k=0}^{n-N-1} \mu^k m_{n-1-k}.
\end{gather*}
Since $\mu\in[0,1)$ and $m_n\rightarrow 0$ it follows that $\eps_n \rightarrow 0$.
\end{proof}

Let us emphasize that we did not use the definition of $(d_n)_n$ nor the fact that the $f_i$'s are affine. Hence the result is valid for \textit{any} pairs of sequences $(K_n)_n$ and $(d_n)_n$ satisfying~\eqref{def:normhutchiter}.

Let us notice that the sequence $(d_n)_n$ depends on $K_0$, so that the two limits $d$ and $L_d$ may also depend on $K_0$. If $d\leqslant\lambda_H$, the asymptotic behavior of $(K_n)_n$ is more delicate to derive directly from the one of $(d_n)_n$. Therefore, in view of Theorem \ref{theo:renorm}, we ask the following questions: Does the sequence $(d_n)_n$ always converge? Does its limit may not depend on $K_0$ or may be smaller than $\lambda_H$?

\subsection[Convergence of $(d_n)_n$]{Convergence of $\boldsymbol{(d_n)_n}$}\label{subsec:cvgdn}

Except for very special cases it is impossible to obtain the exact expression of $d_n$. Therefore we rather seek for bounds for $d_n$ and $d$. Let us begin with a basic result.

\begin{Lemma}\label{lem:bounddn}Let $(d_n)_n$ be the sequence defined in \eqref{eq:defdn}. Then,
\begin{gather}\label{eq:bounddn}
\forall\,n\geqslant 1,\qquad \max_{1\leqslant i \leqslant p}\{\|b_i\|-\|A_i\|\} \leqslant d_n \leqslant \max_{1\leqslant i \leqslant p}\{\|A_i\|+\|b_i\|\}.
\end{gather}
In particular, if $(d_n)_n$ converges to $d$, then $d$ also satisfies~\eqref{eq:bounddn}.
\end{Lemma}

\begin{proof} Let $n\geqslant 1$. One has $f_i(K_n)\subset f_i(B(0,1)) \subset B(b_i,\|A_i\|)$ for all $i\in\{1,\ldots,p\}$. Thus, any $x\in f_i(K_n)$ satisfies $| \|x\| - \|b_i\| | \leqslant \| x- b_i\| \leqslant \|A_i\|$, that is
\begin{gather*}
\|b_i\|-\|A_i\| \leqslant \|x\| \leqslant \|A_i\|+\|b_i\|.
\end{gather*}
Since $d_n = \max\limits_{1\leqslant i \leqslant p} \{\|x\|\colon x\in f_i(K_n)\}$ we obtain \eqref{eq:bounddn}.
\end{proof}

The next result provides non trivial bounds for the possible limit $d$.

\begin{Proposition}\label{prop:majd}If $(d_n)_n$ converges to $d$, then
\begin{gather}\label{eq:propmajd1}
\forall\,i\in\{1,\ldots,p\}, \qquad 0\leqslant \|b_i\| \leqslant \|d\id-A_i\|.
\end{gather}
Moreover, if $i\in\{1,\ldots,p\}$ is such that $d\notin\spec(A_i)$, then
\begin{gather*}
0\leqslant \big\|(d\id-A_i)^{-1}b_i\big\| \leqslant 1.
\end{gather*}
In particular, if $d>\lambda_H$ then
\begin{gather}\label{eq:propmajd3}
\max_{1\leqslant i\leqslant p} \big\{\big\|(d\id-A_i)^{-1}b_i\big\|\big\} \leqslant 1.
\end{gather}
\end{Proposition}

\begin{proof} Let $i\in\{1,\ldots,p\}$ and consider the sequence $(x_n)_{n\geqslant 1}$ defined by $x_1\in K_1$ and $x_{n+1}= \frac1{d_n}(A_i x_n + b_i)$. One has $x_n \in K_n$ and $b_i = d_n x_{n+1} - A_i x_n$. By summation we get
\begin{gather}\label{eq:proofpropmajd1}
nb_i = (d_n x_{n+1} - A_i x_1) + \sum_{k=1}^{n-1} (d_k\id-A_i)x_{k+1}.
\end{gather}
Therefore, for all $n>1$,
\begin{align*}
\|b_i\| & \leqslant \frac1{n}\|d_n x_{n+1} - A_i x_1\| + \frac1{n}\sum_{k=1}^{n-1} \|(d_k\id-A_i)x_{k+1}\| \\
 & \leqslant \frac{2}{n} (d_n + \|A_i\| ) + \left(\frac1{n-1}\sum_{k=1}^{n-1} \|d_k\id-A_i\|\right).
\end{align*}
The first term in the sum above goes to $0$ when $n\to\infty$ and Cesàro's lemma implies that the term into brackets goes to $\|d\id-A_i\|$. That gives~\eqref{eq:propmajd1}.

Now assume that $i$ is such that $d\notin\spec(A_i)$. Then the matrix $M_i=d\id-A_i$ is invertible and~\eqref{eq:proofpropmajd1} yields
\begin{gather*}\label{eq:proofpropmajd2}
n(M_i^{-1} b_i) = M_i^{-1}(d_n x_{n+1} - A_i x_1) + \sum_{k=1}^{n-1} M_i^{-1}(d_k\id-A_i)x_{k+1}.
\end{gather*}
Thus we obtain in a similar way
\begin{align*}
\big\|M_i^{-1}b_i\big\| \leqslant \frac{2}{n}\big\|M_i^{-1}\big\| (d_n + \|A_i\| ) + \left(\frac1{n-1}\sum_{k=1}^{n-1} \big\|M_i^{-1}(d_k\id-A_i)\big\|\right).
\end{align*}
We conclude as above using that $\big\|M_i^{-1}(d_k\id-A_i)\big\|\rightarrow \|\id\| =1$ as $k\rightarrow \infty$.

Finally, since $\|M\|\geqslant\alpha(M)$ holds for every matrix $M$, inequality $d>\lambda_H$ implies that $d\notin\cap_{i=1}^p \spec(A_i)$, which concludes the proof.
\end{proof}

We will now show that \eqref{eq:propmajd3} is an equality when the $A_i$'s are homotheties. Actually, we will prove again~\eqref{eq:propmajd3} but with a very different approach which can be generalized (see Theorem~\ref{theo:cvgfhi}$(i)$). We need the following result. We denote by $\ch(K)$ the convex hull of a non-empty set $K$.

\begin{Lemma}\label{lem:convhull}
Assume that $\|A_i\|<1$ for all $i\in\{1,\ldots,p\}$. Denote by $z_i$ the unique invariant point of $f_i$ and by $L$ the attractor of the $\mathrm{IFS}$ $\{f_1,\ldots,f_p\}$. If $f_j(z_i)\in \ch(\{z_1,\ldots,z_p\})$ for all $i,j\in\{1,\ldots,p\}$, then the convex hull of $L$ is the polytope $\ch(\{z_1,\ldots,z_p\})$.
\end{Lemma}

\begin{proof} Let us write $C=\ch(\{z_1,\ldots,z_p\})$. Let $i\in\{1,\ldots,p\}$. Since $f_i(z_i)=z_i$, one has $z_i\in L \subset \ch(L)$ and then $C \subset\ch(L)$. To prove the reverse inclusion we have to state that $L\subset C$. It is enough to prove that $H(C) \subset C$, i.e., that $f_i(C)\subset C$. So let $z=\sum\limits_{j=1}^p t_j z_j$, $t_j\geqslant 0$ and $\sum\limits_{j=1}^p t_j=1$, a point in $C$. We have
\begin{gather*}
f_i(z) = \sum_{j=1}^p t_j \big ( A_i z_j + b_i \big) = \sum_{j=1}^p t_j f_j(z_i),
\end{gather*}
thus $f_i(z)\in C$.
\end{proof}

\begin{Proposition}\label{prop:limerho}
If $(d_n)_n$ converges to $d$ with $d>\lambda_H$ then $d$ satisfies the inequality
\begin{gather}\label{eq:ineqrho}
\rho\big(\big\{(d\id -A_1)^{-1}b_1,\ldots,(d\id -A_p)^{-1}b_p\big\}\big)\leqslant 1.
\end{gather}
Moreover, if $A_i=\alpha_i \id$ with $\alpha_i\geqslant 0$ for all $i\in\{1,\ldots,p\}$, then \eqref{eq:ineqrho} is an equality. In this case, there is at least one $b_i\neq 0$.
\end{Proposition}

\begin{proof} Let $i\!\in\!\{1,{\ldots},p\}$. First, $d>\lambda_H$ implies that $d\id -A_i$ is invertible and $z_i\!=\!(d\id -A_i)^{-1}b_i$ is the unique invariant point of $\frac1{d}f_i$. Secondly, $d>\lambda_H$ implies that $(K_n)_n$ converges to $L_d$ so $z_i\in L_d$. Therefore $\{z_1,\ldots,z_p\} \subset L_d$, and, by monotonicity, $\rho(\{z_1,\ldots,z_p\})\leqslant \rho(L_d)=1$. That gives~\eqref{eq:ineqrho}.

Now, if all the $A_i$'s are homotheties, one has
\begin{gather*}
\forall\,j\in\{1,\ldots,p\},\qquad \frac{f_j}{d}(z_i)= \frac{\alpha_j}{d}\,z_i + \left(1-\frac{\alpha_j}{d}\right)z_j.
\end{gather*}
Since $0\leqslant \alpha_j < d$, one has $\frac{f_j}{d}(z_i)\in \ch(\{z_1,\ldots,z_p\})$. Thus, it follows from Lemma~\ref{lem:convhull} applied to the IFS $\big\{\frac1{d}f_1,\ldots,\frac1{d}f_p\big\}$ that $\ch(\{z_1,\ldots,z_p\})=\ch(L_d)$. Since $\rho(\ch(K))=\rho(K)$ for all $K\in{\cal K}$, we obtain
\begin{gather*}
1=\rho(L_d)=\rho(\ch(L_d))=\rho(\ch(\{z_1,\ldots,z_p\}))=\rho(\{z_1,\ldots,z_p\}),
\end{gather*}
hence \eqref{eq:ineqrho} becomes an equality. Finally, if $b_i=0$ for all $i\in\{1,\ldots,p\}$ then the left-hand side of \eqref{eq:ineqrho} is zero, hence a contradiction.
\end{proof}

Notice that using the stability property of $\rho$, \eqref{eq:ineqrho} gives \eqref{eq:propmajd3}.

We conclude now by giving another non trivial bounds for $d$ valid for a particular class of IFS's. The next result is only a rephrasing of Theorems~2 and~3 in \cite{barvin11}.

\begin{Proposition}\label{prop:limJSR} Assume that the $A_i$'s have no common invariant subspaces except~$\{0\}$ and~${\mathbb R}^D\!$. If $(K_n)_n$ converges to $K\in{\cal K}$, then $(d_n)_n$ converges to
\begin{gather*}\label{eq:limJSR}
d=\max_{1\leqslant i \leqslant p} \rho(f_i(K)) \geqslant \sigma_{{\cal M}}
\end{gather*}
and equality holds if $b_i=0$ for all $i\in\{1,\ldots,p\}$.
\end{Proposition}

The determination of $\sigma_{{\cal M}}$ is delicate but the basic estimates
\begin{gather*}\label{eq:basicestimJSR}
\max_{1\leqslant i\leqslant p}\{\alpha(A_i)\} \leqslant \sigma_{{\cal M}} \leqslant \max_{1\leqslant i\leqslant p}\{\|A_i\|\} = \lambda_H
\end{gather*}
always hold (see \cite{theys}). In particular for homotheties, i.e., when $A_i=\alpha_i \id$ with $\alpha_i\geqslant 0$, one obtains $\sigma_{{\cal M}}=\lambda_H = \max\limits_{1\leqslant i\leqslant p}\{\alpha_i\}$. Unfortunately, this simple case does not fulfill the hypotheses of Proposition~\ref{prop:limJSR}.

\subsection{Case of homotheties}\label{subsec:radhomo}

We can give a complete answer when all the $A_i$'s are homotheties: the sequence $(K_n)_n$ always converges and its limit may be explicited. First, we show that $(d_n)_n$ converges and we give the possible value for its limit~$d$.

\begin{Lemma}\label{lem:limefhid} Assume that $A_i=\alpha_i \id$ with $\alpha_i\geqslant 0$ for all $i\in\{1,\ldots,p\}$. Let $j$ be an index such that $\alpha_j=\lambda_H=\max\limits_{1\leqslant i\leqslant p}\{\alpha_i\}$. Then $(d_n)_n$ converges to a number $d>0$. If $d=\alpha_j$ then $b_j=0$ else $d\neq\alpha_j$ and satisfies
\begin{gather*}\label{eq:ineqfhif}
d= \begin{cases}
 \displaystyle\max_{1\leqslant i\leqslant p}\{\alpha_i+\|b_i\|\} & \text{if $d>\alpha_j$}, \\
 \alpha_j-\|b_j\| & \text{if $d<\alpha_j$.}
 \end{cases}
\end{gather*}
\end{Lemma}

\begin{proof}For all $n\geqslant 1$ we can find $y_n\in K_n$ and $i_n\in\{1,\ldots,p\}$ such that $d_n=\|\alpha_{i_n} y_n + b_{i_n}\|$. Then, $u_n=\frac1{d_n}(\alpha_{i_n} y_n + b_{i_n})$ satisfies $u_n\in K_{n+1}$ and $\|u_n\|=1$. Since $\|y_n\|\leqslant 1$ we obtain
\begin{gather*}
d_{n+1}\geqslant \|\alpha_{i_n} u_n + b_{i_n}\| \geqslant \|\alpha_{i_n} u_n + d_n u_n\| - \|d_n u_n \negthickspace - b_{i_n}\|
 = (\alpha_{i_n}\negthickspace + d_n)\|u_n\| - \alpha_{i_n}\|y_n\|\geqslant d_n.
\end{gather*}
Thus $(d_n)_{n\geqslant1}$ is increasing and bounded (see \eqref{eq:bounddn}), so it converges. Let $d$ be its limit.

For all $n\geqslant 1$, choosing $x_n\in K_n$ such that $\|x_n\|=1$ we get
\begin{gather}\label{eq:minordj}
d\geqslant d_n\geqslant \|\alpha_j x_n + b_j\| \geqslant \|\alpha_j x_n\| - \|b_j\| = \alpha_j - \|b_j\|.
\end{gather}

Inequality \eqref{eq:propmajd1} with $i=j$ writes $\|b_j\|\leqslant |d-\alpha_j|$ so $d=\alpha_j$ implies $b_j=0$. If $d<\alpha_j$ then $d\leqslant \alpha_j-\|b_j\|$. In addition with~\eqref{eq:minordj} we obtain $d=\alpha_j-\|b_j\|$.

If $d>\alpha_j$, it follows from Proposition~\ref{prop:limerho} that $d$ is a solution of $\max\limits_{1\leqslant i \leqslant p} \frac{\|b_i\|}{|t-\alpha_i|}= 1$. We can consider only the $b_i\neq 0$. Then, since $d>\lambda_H$ and the functions $t\mapsto \frac{\|b_i\|}{|t-\alpha_i|}$ are strictly decreasing on $(\lambda_H,+\infty)$, the unique solution is $\max\limits_{1\leqslant i\leqslant p}\{\alpha_i+\|b_i\|\}$.
\end{proof}

We can state now the precise result. We denote by $\cl(K)$ the closure of a non-empty set $K$.

\begin{Theorem}\label{theo:radius}Assume that $A_i=\alpha_i \id$ with $\alpha_i\geqslant 0$ for all $i\in\{1,\ldots,p\}$. Let $j$, $k$ two indices such that $\alpha_j =\max\limits_{1\leqslant i\leqslant p}\{\alpha_i\}$ and $\alpha_k+\|b_k\|=\max\limits_{1\leqslant i\leqslant p}\{\alpha_i+\|b_i\|\}$. Then, for all $K_0\in {\cal K}$, the sequence~$(K_n)_n$ converges to a set $K\in{\cal K}$. Precisely,
\begin{enumerate}\itemsep=0pt
\item[$(i)$] if $b_j\neq 0$ then
\begin{enumerate}\itemsep=0pt
\item[$(a)$] either $\alpha_j-\|b_j\|>0$, $f_i\big({-}\frac1{\|b_j\|}b_j\big)=(\alpha_j-\|b_j\|)\big({-}\frac1{\|b_j\|}b_j\big)$ for all $i\in\{1,\ldots,p\}$ and $K_0=H_{\rho}^{-1}\big(\big\{{-}\frac1{\|b_j\|}b_j\big\}\big)$: in this case $K=\big\{{-}\frac1{\|b_j\|}b_j\big\}$,
\item[$(b)$] or else $K$ does not depend on $K_0$: it is the attractor $L_d$ with $d=\alpha_k+\|b_k\|$ and $\frac1{\|b_k\|}b_k \in L_d$;
\end{enumerate}
\item[$(ii)$] if $b_j=0$ then
\begin{enumerate}\itemsep=0pt
\item[$(a)$] either $\alpha_j\geqslant \alpha_k+\|b_k\|$ and then $K=\cl\big({\bigcup}_{n\geqslant 1} K_n\big)$,
\item[$(b)$] or else $\alpha_j<\alpha_k+\|b_k\|$ and then $K$ does not depend on $K_0$: it is the attractor $L_d$ with $d=\alpha_k+\|b_k\|$ and $\frac1{\|b_k\|}b_k \in L_d$.
\end{enumerate}
\end{enumerate}
\end{Theorem}

\begin{proof}(i) Assume that $b_j\neq 0$. Hence by Lemma \ref{lem:limefhid} we have $d\neq \alpha_j$.

(a) Suppose first that $d<\alpha_j$. Then, use of Lemma \ref{lem:limefhid} again shows that $d=\alpha_j-\|b_j\|$. In particular $\alpha_j-\|b_j\|\neq 0$ and $\alpha_j\neq 0$. Moreover, it follows from~\eqref{eq:minordj} that $d_n=d$ for all $n\geqslant1$. Let $x_n\in K_n$ and consider the sequence $(x_{n+k})_k$ defined by $x_{n+k+1}=\frac1{d_{n+k}}(\alpha_j x_{n+k} + b_j)=\frac1{d}f_j(x_{n+k})$ for all $k\geqslant0$. Notice that $x_{n+k} \in K_{n+k}$ so in particular $\|x_{n+k}\|\leqslant 1$. Let us introduce $u=-\frac1{\|b_j\|}b_j$ the unique point such that $f_j(u)=du$. Noticing that $x_{n+k+1}-u=\frac{\alpha_j}{d}(x_{n+k}-u)$ we obtain by induction that
\begin{gather*}
\forall\,k\geqslant0, \qquad 2\geqslant\|x_{n+k}-u\|=\left(\frac{\alpha_j}{d}\right)^{k}\|x_n-u\|\geqslant0.
\end{gather*}
Since $d<\alpha_j$ we must have $\|x_n-u\|=0$. It follows that $K_n=\{u\}$ for all $n\geqslant1$. Thus $f_i(u)=du$ for all $i\in\{1,\ldots,p\}$ and $K=\{u\}$. Therefore conditions of (a) are all fulfilled. Conversely, if they are satisfy we have obviously $K_n=K_1$ for all $n\geqslant0$ and the result.

(b) Suppose now that $d>\alpha_j$. Then Lemma \ref{lem:limefhid} implies that $d=\alpha_k+\|b_k\|$. Since $d>\lambda_H$ the convergence to $L_d$ follows from Theorem \ref{theo:renorm}. Since $L_d$ is the attractor of the IFS $\big\{\frac1{d}f_1,\ldots,\frac1{d}f_p\big\}$, it contains the invariant point $z_k$ of $\frac1{d}f_k$ which is $z_k=\frac1{\|b_k\|}b_k$.

\vspace{0.1cm}(ii) Assume that $b_j=0$. Hence by \eqref{eq:minordj} we have $d\geqslant\alpha_j$.

(a) Suppose first that $\alpha_j\geqslant \alpha_k+\|b_k\|$. Then it follows from Lemma \ref{lem:bounddn} that $d=\alpha_j$ and then by~\eqref{eq:minordj} that $d_n=d$. Therefore, for all
$n\geqslant1$,
\begin{gather*}
K_{n+1} = \frac1{\alpha_j} \bigcup_{i=1}^p f_i(K_n) = K_n \cup \bigcup_{i\neq j}^p f_i(K_n).
\end{gather*}
Thus $(K_n)_{n\geqslant1}$ is increasing. Since it is bounded it converges to $\cl(\bigcup_{n\geqslant 1} K_n)$.

(b) Suppose now that $\alpha_j<\alpha_k+\|b_k\|$. Assume that $d=\alpha_j$. Then inequality~\eqref{eq:propmajd1} with $i=k$ yields either $d\geqslant \alpha_k+\|b_k\|$ or $d\leqslant \alpha_k-\|b_k\|$. This latter being the unique possibility we obtain $\alpha_k\leqslant \alpha_j\leqslant \alpha_k-\|b_k\|$. Thus $b_k=0$ and $\alpha_j=\alpha_k$ which is a contradiction. Therefore $d>\alpha_j$ and we conclude along the same lines as for~(i)(b).
\end{proof}

Let us note that, when $D=1$, the unit sphere being finite, we can prove that $(d_n)_n$ is always stationary.

\begin{Example}\label{exa:limhomo1} Let us consider the IFS $\{f_1,f_2,f_3\}$ where the $f_i\colon {\mathbb R}^2\to {\mathbb R}^2$ are given by $f_i(x)=2x+b_i$ with
\begin{gather*} b_1 = \left[\begin{matrix}
0 \\
0
\end{matrix}\right] ,\qquad b_2 = \left[ \begin{matrix}
2 \\
0
\end{matrix}\right]
\qquad \text{and}\qquad b_3 = \left[ \begin{matrix}
0 \\
2 \\
\end{matrix}\right].\end{gather*}
Then, for all $K_0\in {\cal K}$, the sequence $(K_n)_n$ converges to the attractor of the IFS $\big\{\frac1{4}f_1,\frac1{4}f_2,\frac1{4}f_3\big\}$. It is a classical Sierpinski gasket (see Fig.~\ref{fig:limhomo}(a)).
\end{Example}

\begin{proof}We apply Theorem \ref{theo:radius} with $\max\limits_{1\leqslant i\leqslant 3}\{\alpha_i+\|b_i\|\}=4$ and $\max\limits_{1\leqslant i\leqslant 3}\{\alpha_i\}=2$ (notice here that indices $k$ and $j$ are not unique). Whatever is the choice of $k$ and $j$, we are here in the case (b). We have $d=4$ and $(1,0)\in L_d$, $(0,1)\in L_d$.
\end{proof}

\begin{Example}\label{exa:limhomo2}Let us consider the IFS $\{f_1,f_2,f_3\}$ where the $f_i\colon {\mathbb R}^2\to {\mathbb R}^2$ are given by $f_1(x)=6x+b_1$, $f_2(x)=4x+b_2$ and $f_3(x)=3x+b_3$ with
\begin{gather*} b_1 = \left[\begin{matrix}
0 \\
0
\end{matrix}\right],\qquad b_2 = \left[ \begin{matrix}
0 \\
1
\end{matrix}\right]
\qquad \text{and}\qquad b_3 = \left[ \begin{matrix}
-2 \\
2
\end{matrix}\right].\end{gather*}
Then, for all $K_0\in {\cal K}$, the sequence $(K_n)_n$ is increasing and converges to the set $\cl\big(\bigcup_{n\geqslant 1} K_n\big)$ (see Fig.~\ref{fig:limhomo}(b) for $K_0=\{0\}\times[0,1]$).
\end{Example}

\begin{proof}We apply Theorem \ref{theo:radius} with $\max\limits_{1\leqslant i\leqslant 3}\{\alpha_i+\|b_i\|\}=5$ and $\max\limits_{1\leqslant i\leqslant 3}\{\alpha_i\}=6$. Thus we are here in the case~(ii)(a).
\end{proof}

\begin{figure}[!h]\centering
\begin{minipage}{75mm}\centering
\includegraphics[scale=0.27]{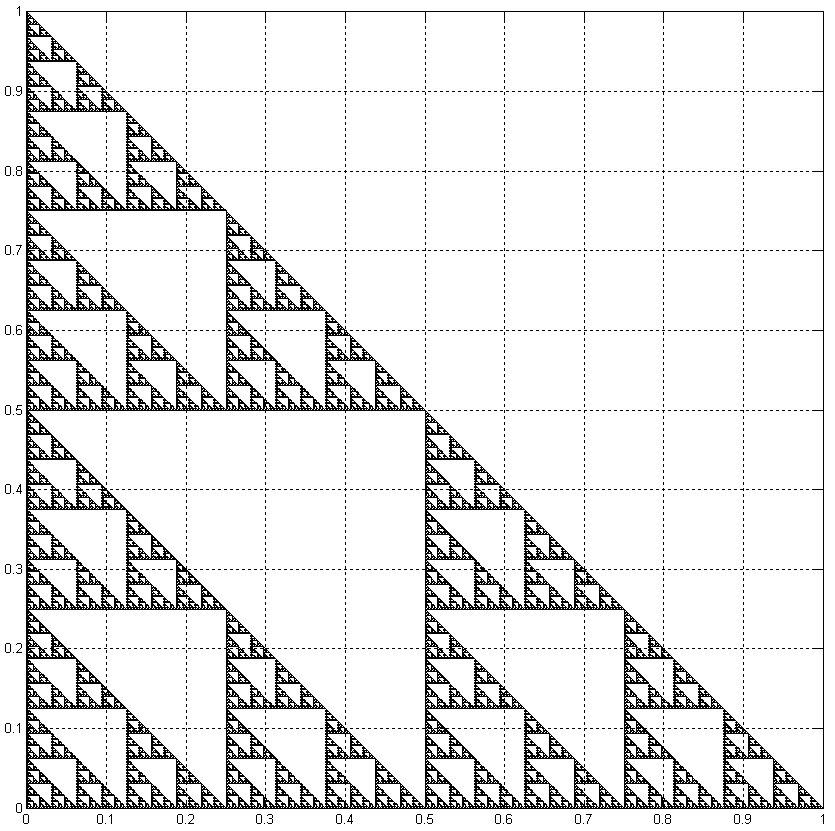}\\
{\footnotesize (a) a classical Sierpinski gasket}
\end{minipage}
\qquad
\begin{minipage}{75mm}\centering
\includegraphics[scale=0.275]{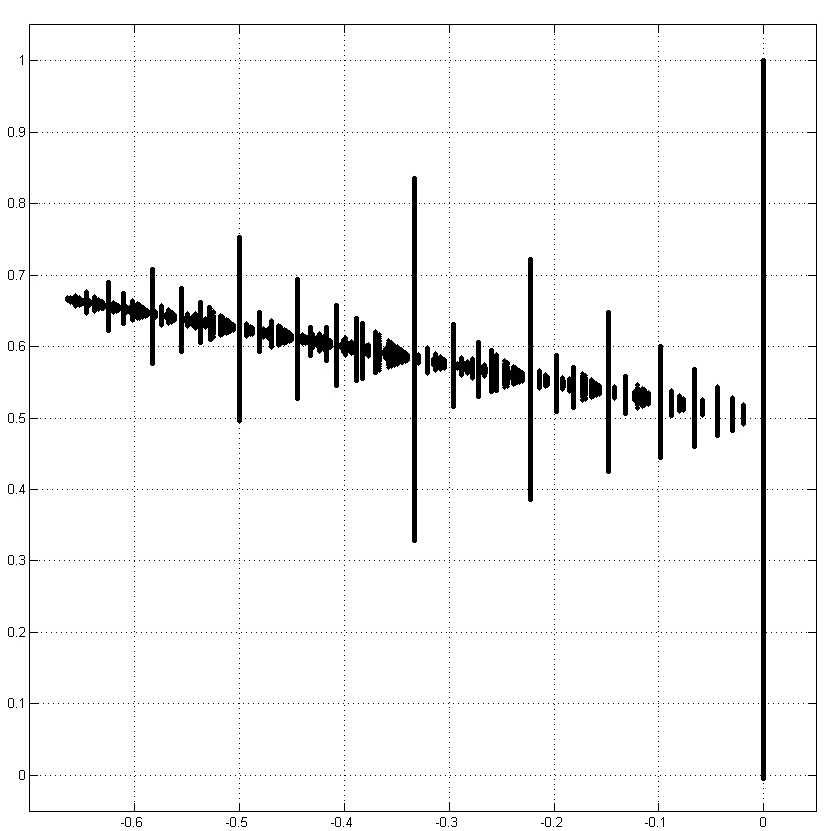}\\
{\footnotesize (b)~a union of vertical segments}
\end{minipage}
\caption{The limit set $K$ obtained by renormalizing the IFS $\{f_1,f_2,f_3\}$ with the radius function~$\rho$. Maps $f_1$, $f_2$, $f_3$ are given in Example~\ref{exa:limhomo1} for (a) and Example~\ref{exa:limhomo2} for~(b).}\label{fig:limhomo}
\end{figure}

We can observe that in the previous theorem the asymptotics of $(d_n)_n$ was given by the points of $B(0,1)$ whose image by the $f_i$'s have the largest norm. Exploiting this remark we can obtain a more general result assuming only one function $f_i$ has a homothety linear part but which is responsible of large norms.

\begin{Proposition}\label{prop:homosolo} Assume that $f_p(x)=\alpha x+b$ with $\alpha\geqslant 0$ and that
\begin{gather*}
\forall\,i\in\{1,\ldots,p-1\},\qquad f_i(B(0,1)) \subset B(0,|\alpha-\|b\||).
\end{gather*}
Then, for all $K_0\in {\cal K}$, the sequence $(K_n)_n$ converges to a set $K\in{\cal K}$. Precisely,
\begin{enumerate}\itemsep=0pt
\item[$(i)$] if $b\neq 0$ then
\begin{enumerate}\itemsep=0pt
\item[$(a)$] either $\alpha-\|b\|>0$, $f_i\big({-}\frac1{\|b\|}b\big)=(\alpha-\|b\|)\big({-}\frac1{\|b\|}b\big)$ for all $i\in\{1,\ldots,p\}$ and $K_0=H_{\rho}^{-1}\big(\big\{{-}\frac1{\|b_j\|}b_j\big\}\big)$: in this case $K=\big\{{-}\frac1{\|b\|}b\big\}$,
\item[$(b)$] or else $K$ does not depend on $K_0$: it is the attractor $L_d$ with $d=\alpha+\|b\|$ and $\frac1{\|b\|}b \in L_d$;
\end{enumerate}
\item[$(ii)$] if $b=0$ then $K=\cl\big({\bigcup}_{n\geqslant 1} K_n\big)$.
\end{enumerate}
\end{Proposition}

\begin{proof} Actually the proof is very similar to the one of Theorem~\ref{theo:radius} thus we will only detail the key-points.
The first point is to show that $d_n$ is always obtained with the function~$f_p$. Indeed, let $n\geqslant 1$ and $x_n\in K_n$ with $\|x_n\|=1$. Then, for all $i\in\{1,\ldots,p-1\}$,
\begin{gather*}
\|f_p(x_n)\| = \|\alpha x_n+ b\| \geqslant | \|\alpha x_n\| - \|b\| | = | \alpha - \|b\| | \geqslant \rho(f_i(B(0,1))) \geqslant \rho(f_i(K_n)).
\end{gather*}
It follows that $d_n=f_p(y_n)$ for some $y_n\in K_n$ and $d_n\geqslant | \alpha - \|b\| |$.

Considering now $u_n=\frac1{d_n}(\alpha y_n + b)$ we get $u_n\in K_{n+1}$, $\|u_n\|=1$ and, since $\|y_n\|\leqslant 1$,
\begin{align*}
d_{n+1}\geqslant \|\alpha u_n + b\| \geqslant \|\alpha u_n + d_n u_n\| - \|d_n u_n - b\| = (\alpha + d_n)\|u_n\| - \alpha\|y_n\|\geqslant d_n.
\end{align*}
Thus $(d_n)_{n\geqslant1}$ is increasing, bounded, so it converges. Let $d$ be its limit.

In particular, we have proved that
\begin{gather}\label{eq:ineghomosolo}
\alpha - \|b\| \leqslant | \alpha - \|b\| | \leqslant d_n \leqslant d \leqslant \alpha + \|b\|.
\end{gather}

Besides, inequality \eqref{eq:propmajd1} with $i=p$ writes $\|b\|\leqslant |d-\alpha|$ hence either $d\geqslant \alpha+\|b\|$ or $d\leqslant \alpha-\|b\|$. Finally, $d\in\{\alpha-\|b\|,\alpha+\|b\|\}$.

(i) Assume that $b\neq 0$. Hence $d\neq \alpha$.

(a) Suppose first that $d=\alpha-\|b\|$. This case is similar to the case (i)(a) of Theorem \ref{theo:radius}.

(b) Suppose now that $d=\alpha+\|b\|$. First, $d>\alpha$. Next, for $i\in\{1,\ldots,p-1\}$,
\begin{gather*}
\|A_i\| \leqslant \sup_{\|x\|=1}\{ \|A_i x +b_i \| \} \leqslant | \alpha-\|b\| | < \alpha + \|b\|.
\end{gather*}
It follows that $d>\lambda_H$ and we conclude as for the case (i)(b) of Theorem \ref{theo:radius}.

(ii) Assume that $b=0$. It follows from \eqref{eq:ineghomosolo} that $d_n=d=\alpha$ for all $n\geqslant1$. This case is then similar to the case (ii)(a) of Theorem \ref{theo:radius}.
\end{proof}

\section{Results for linear IFS's}\label{sec:linear}

We suppose in this section that the IFS consists in $p\geqslant1$ linear maps $f_i\colon {\mathbb R}^D\rightarrow {\mathbb R}^D$ defined by $f_i(x)=A_ix$. We still denote by ${\cal M}=\{A_1,\ldots,A_p\}$ the set of their canonical matrices.

\subsection{New strategy}\label{subsec:cvge}

If $(d_n)_n$ converges to $d$, it follows from Lemma \ref{lem:bounddn} that $d\leqslant \lambda_H$ so we cannot apply Theo\-rem~\ref{theo:renorm}. Actually the convergence of $(d_n)_n$ may not imply the convergence of $(K_n)_n$. Consider for example $D=2$ and the functions $f_1$, $f_2$ defined by
\begin{gather*} A_1=\left[\begin{matrix}
2 & 0 \\
0 & -3
\end{matrix}\right]\qquad \text{and}\qquad A_2=\left[\begin{matrix}
1 & 0 \\
0 & -1
\end{matrix}\right].\end{gather*}
If $K_0=\{(1,0)\}$ then $K_n=\big\{ \big(2^{-k},0\big)\colon 0\leqslant k \leqslant n\big\}$ hence converges to $\cl\big(\bigcup_{n\geqslant0} K_n\big)$ and $(d_n)_n$ is constant to $2<3=\lambda_H$. If $K_0=\{(0,1)\}$ then $K_n=(-1)^n\big\{ \big(0,3^{-k}\big) \colon 0\leqslant k \leqslant n\big\}$ hence diverges but~$(d_n)_n$ is constant to~$\lambda_H$.

Therefore we adopt here a new strategy, taking advantage of both the linearity of the $f_i$'s and the homogeneity of $\rho$.

\begin{Proposition}\label{prop:linred}Let $K_0\in{\cal K}$. Assume that there exists $d>0$ such that $\big(H_d^n(K_0)\big)_n$ converges to a set $L\in{\cal K}$ with $\rho(L)\neq0$. Then,
\begin{enumerate}\itemsep=0pt
\item[$(i)$] $(d_n)_n$ converges to $d$ and $d\leqslant \sigma_{{\cal M}}$,
\item[$(ii)$] $(K_n)_n$ converges to $K=\frac1{\rho(L)}\,L$.
\end{enumerate}
\end{Proposition}

\begin{proof} Let $n\geqslant 1$. We obtain by linearity
\begin{gather*}\label{eq:linprod}
K_n = \frac1{d_0\cdots d_{n-1}} H^n(K_0).
\end{gather*}
Since $\rho(K_n)=1$, it follows by homogeneity that $\rho\big(H^n (K_0)\big) = d_0\cdots d_{n-1}$. Using linearity again we observe that $H_d^n(K_0)= \frac1{d^n} H^n(K_0)$. Thus $\rho\big(H_d^n(K_0)\big) = \prod\limits_{k=0}^{n-1}\frac{d_k}{d}$. By hypothesis, this last sequence is a proper convergent product, hence $\frac{d_k}{d} \to 1$. Moreover, if $d>\sigma_{{\cal M}}$ then the joint spectral radius of $\big\{\frac1{d}A_1,\ldots,\frac1{d}A_p\big\}$ is $\frac{\sigma_{{\cal M}}}{d}<1$, hence $\big(H_d^n(K_0)\big)_n$ converges to $\{0\}$ (see~\cite{barvin11}) which is a contradiction. Thus we get $(i)$. Finally,
\begin{gather*}\label{eq:prooflinred}
K_n = \frac1{\rho\big(H^n (K_0)\big)} H^n (K_0) = \frac1{\rho\big(H_d^n (K_0)\big)} H_d^n (K_0).
\end{gather*}
Hypotheses and continuity of $\rho$ allow us to take the limit in the right-hand side above. That gives~(ii).
\end{proof}

As we saw in Proposition \ref{prop:limJSR} if such a $d$ exists then $d=\sigma_{{\cal M}}$ for a large class of IFS's. We can expect that this is true in general. However, the first example in this section shows that the strict inequality is possible even for very simple IFS's. Actually the hypothesis on the common invariant subspaces of the $A_i$'s is essential. Notice that the $A_i$'s share a common non trivial invariant subspace if and only if there exists an invertible matrix $P$ such that, for all $i\in\{1,\ldots,p\}$,
\begin{gather*} A_i = P\left[\begin{matrix}
A'_i & M_i \\
0 & A''_i
\end{matrix}\right]P^{-1} \end{gather*}
with $A'_i$ and $A''_i$ square and some matrix $M_i$ (see~\cite{theys}). This in particular the case of diagonal matrices, where the numerous invariant spaces will provide very special behaviors for $(K_n)_n$. In the rest of this section, we will look at such IFS's focusing on the convergence of $\big(H_d^n(K_0)\big)_n$ especially for $d=1$.

\subsection{LCP sets of matrices}\label{subsec:linlcpsets}

We say that ${\cal M}$ is a \textit{left convergent product set of matrices} (LCP set in short) if the infinite products $A_{i_n}\cdots A_{i_1}$ converge for all sequences $(i)=(i_1,i_2,\dots)\in{\cal I}=\{1,\ldots,p\}^{\infty}$. In this case, we set $A_{(i)} = \lim\limits_{n\to \infty} A_{i_n}\cdots A_{i_1}$ (see~\cite{daub1,hart}). The theory of LCP sets was popularized in the 90s (see~\cite{daub1}) and it is still of interest nowadays (see~\cite{hart}) for example in the study of inhomogeneous Markov chains (see, e.g.,~\cite{sene}). One can always associate a canonical IFS with a LCP set. The next result gives sufficient conditions to obtain its convergence.

\begin{Lemma}\label{lem:cvglcp}Assume that
\begin{enumerate}\itemsep=0pt
\item[$(i)$] ${\cal M}$ is a $\mathrm{LCP}$ set,
\item[$(ii)$] There exists a sequence $(\eps_n)_n$ of positive numbers such that $\eps_n \to 0$ and
\begin{gather}\label{eq:cvglcp}
\forall\,(i)=(i_1,i_2,\dots)\in{\cal I},\qquad \forall\,n\geqslant 1, \qquad \|A_{(i)}-A_{i_n}\cdots A_{i_1}\|\leqslant \eps_n.
\end{gather}
\end{enumerate}
Then, $\big(H^n(K_0)\big)_n$ converges for all $K_0\in {\cal K}$ to the limit set
\begin{gather}\label{eq:limlcp}
L= \cl\left(\bigcup_{(i)\in{\cal I}} A_{(i)}(K_0)\right).
\end{gather}
\end{Lemma}

\begin{proof} Let us write $K_n'=H^n(K_0)$ and $L'= \bigcup_{(i)\in{\cal I}} A_{(i)}(K_0)$. Hypothesis (i) implies that ${\cal M}$ is product bounded (see \cite{berg}), then there exists $R>0$ such that $\|A_{(i)}\|\leqslant R$ for all $(i)\in{\cal I}$. Since $K_0$ is compact, it follows that $L'$ is bounded, hence $L$ is compact. We claim that $\dha(K_n',L')\leqslant C\eps_n$ for all $n\geqslant1$, $C>0$. Let $n\geqslant1$ be fixed. We have
\begin{gather*}
K_n'=\{A_{i_n}\cdots A_{i_1}(x_0) \colon x_0\in K_0 \text{ and } 1\leqslant i_1,\ldots,i_n\leqslant p\}.
\end{gather*}
Let $x'\in L'$. One has $x'=A_{(i)}(x_0)$ with $x_0\in K_0$ and $(i)=(i_1,i_2,\dots,i_n,\dots)\in{\cal I}$. Let $x=A_{i_n}\cdots A_{i_1}(x_0)$. One has $x\in K_n'$ and $\|x'-x\| \leqslant \|A_{(i)}-A_{i_n}\cdots A_{i_1}\| \|x_0\| \leqslant C\eps_n$ where $C=\rho(K_0)$. Thus $L'\subset K_n'(C\eps_n)$. We prove in a similar way that $K_n'\subset L'(C\eps_n)$, hence $\dha(K_n',L')\leqslant C\eps_n$. It follows that $\dha(K_n',L')\to 0$ and, since $\dha(K_n',L')=\dha(K_n',L)$, that $K_n'\to L$.
\end{proof}

We illustrate this result with the family of positive stochastic matrices in ${\mathbb R}^2$. In this case, we can give a precise description of the limit set $L$. Let us recall that a positive stochastic matrix is a matrix whose rows consist of positive real numbers, with each row summing to~$1$.

\begin{Proposition}\label{prop:limstoc2} Let $p\geqslant1$ linear maps $f_i\colon {\mathbb R}^2\to{\mathbb R}^2$ of the form $f_i(x)=A_ix$ where $A_i$ is a~positive stochastic matrix. Then, for all $K_0\in {\cal K}$, the sequence $\big(H^n(K_0)\big)_n$ converges to the set
\begin{gather*}
L=\cl\left(\bigcup_{v_0\in K_0} \big\{(x,x) \colon x\in h_{v_0}(\Gamma)\big\}\right),
\end{gather*}
where $h_{v_0}\colon {\mathbb R}\to {\mathbb R}$ is an affine map which depends on $v_0$ and $f_i$, and $\Gamma$ is the attractor of an $\mathrm{IFS}$ $\{g_1,\ldots,g_p\}$ where $g_i\colon {\mathbb R}\to {\mathbb R}$ is an affine map which only depends on $f_i$.
\end{Proposition}

\begin{proof} We will apply Lemma \ref{lem:cvglcp}. First, since each product $A_{i_n}\cdots A_{i_1}$ contains a positive stochastic matrix then it converges (see~\cite{bru}) and ${\cal M}$ is a LCP set. Next, there exists a matrix~$P$ of the form
\begin{gather*}
P = \left[\begin{matrix}
1 & u \\
1 & v
\end{matrix}\right] \qquad \text{with} \quad u,v\in {\mathbb R},
\end{gather*}
such that, for all $i\in\{1,\ldots,p\}$, $A_i=P T_i P^{-1}$ where $T_i$ is a matrix of the form
\begin{gather*}
T_i = \left[\begin{matrix}
1 & a_i \\
0 & b_i
\end{matrix}\right] \qquad \text{with} \quad a_i\in {\mathbb R} \qquad \text{and}\qquad b_i\in[0,1).
\end{gather*}
Notice that it is also proved in \cite{daub1} that $\{T_1,\ldots,T_p\}$ is a LCP set with a continuous limit function. Here we want more and describe precisely the limit set of matrices. Let us define $g_i\colon {\mathbb R}\to {\mathbb R}$ by $g_i(x)=b_ix+a_i$ and set $b=\max\limits_{1\leqslant i\leqslant p} b_i<1$. We obtain by induction that, for all sequences of indices $i_1,\ldots,i_n$, $n\geqslant 2$,
\begin{gather*}
T_{i_n}\cdots T_{i_1} = \left[\begin{matrix}
1 & (g_{i_1}\circ g_{i_2} \circ \cdots \circ g_{i_{n-1}})(a_{i_n}) \\
0 & b_{i_1}\cdots b_{i_n}
\end{matrix}\right].
\end{gather*}
Hence, considering the contractive Hutchinson operator $G$ associated with the IFS $\{g_1,\ldots,g_p\}$, its attractor $\Gamma$, and the orbit $\big(G^n(A)\big)_n$ of the compact set $A=\{a_1,\ldots,a_p\}$, we obtain, for all $n\geqslant 2$,
\begin{gather*}\label{eq:proofpropstoch}
\big\{ T_{i_n}\cdots T_{i_1}\big\}_{1\leqslant i_1,\ldots,i_n\leqslant p} = \left\{ \left[\begin{matrix}
1 & c_{i_1 \cdots i_n} \\
0 & b(c_{i_1 \cdots i_n})
\end{matrix}\right], \, c_{i_1 \cdots i_n}\in G^{n-1}(A)\right\}
\end{gather*}
with $0\leqslant b(c_{i_1 \cdots i_n}) \leqslant b^n$ and $\dd(c_{i_1 \cdots i_n},\Gamma)\leqslant C b^n$ for a constant $C>0$ (see~\cite{hutch1}).

It follows first that, for all $(i)\in{\cal I}$ and all $n\geqslant1$,
\begin{gather*}
\|A_{(i)}-A_{i_n}\cdots A_{i_1}\|=\big\|P\big(T_{(i)}-T_{i_n}\cdots T_{i_1}\big)P^{-1}\big\|\leqslant \|P\| \|T_{(i)}-T_{i_n}\cdots T_{i_1}\| \big\|P^{-1}\big\| \leqslant C' b^n
\end{gather*}
with $C'>0$. Hence \eqref{eq:cvglcp} and all the hypotheses of Lemma~\ref{lem:cvglcp} are satisfied with $\eps_n=C'b^n$.

Moreover, letting $n$ goes to $\infty$ we obtain the following set of limit matrices:
\begin{gather*}
\left\{ T_{(i)}\right\}_{(i)\in{\cal I}} = \left\{ \left[\begin{matrix}
1 & c \\
0 & 0
\end{matrix}\right], \, c\in \Gamma\right\}.
\end{gather*}

Therefore, if $v_0=(x_0,y_0)\in K_0$ we get
\begin{gather*}
A_{(i)}(v_0) = \left(P \left[\begin{matrix}
1 & c \\
0 & 0
\end{matrix}\right] P^{-1}\right) \left[\begin{matrix}
x_0 \\
y_0
\end{matrix}\right] = \left[\begin{matrix}
\dfrac{y_0-x_0}{v-u} c + \dfrac{x_0v-y_0u}{v-u} \vspace{1mm}\\
\dfrac{y_0-x_0}{v-u} c + \dfrac{x_0v-y_0u}{v-u}
\end{matrix}\right].
\end{gather*}
The result follows by taking $h_{v_0}(x)=\frac{y_0-x_0}{v-u} x + \frac{x_0v-y_0u}{v-u}$ and using~\eqref{eq:limlcp}.
\end{proof}

Notice that if $K_0\subset \Span\{(1,1)\}$ then $f_i(v_0)=v_0$ and $L=K_0$. Actually, $\Span\{(1,1)\}$ is a~common invariant space of the $A_i$'s so that $K_n=K_0$ for all $n\geqslant 0$. In particular, one cannot apply Proposition~\ref{prop:limJSR} but it follows from the decomposition of the $A_i$'s that $d=1=\sigma_{{\cal M}}$.

\begin{Example}\label{exa:limstoc}Let us consider the IFS $\{f_1,f_2\}$ where the $f_i\colon {\mathbb R}^2\to {\mathbb R}^2$ are the linear maps given by their canonical matrices
\begin{gather*} A_1 = \frac1{4}\left[\begin{matrix}
1+3a & 3-3a \\
1-a & 3+a
\end{matrix}\right] \qquad \text{and}\qquad A_2 = \frac1{2}\left[ \begin{matrix}
1+a & 1-a \\
1-a & 1+a
\end{matrix}\right] \end{gather*}
with $0<a<1$. With $P=\left[\begin{smallmatrix}
1 & 3 \\
1 & -1
\end{smallmatrix}\right]$ one obtains
\begin{gather*} T_1 = \left[\begin{matrix}
1 & 0 \\
0 & a
\end{matrix}\right] \qquad \text{and}\qquad T_2 = \left[ \begin{matrix}
1 & 1-a \\
0 & a
\end{matrix}\right]. \end{gather*}
Thus $g_1(x)=ax$, $g_2(x)=ax + (1-a)$ and $\Gamma$ is the Cantor set $\Gamma_a$ when $0<a<\frac1{2}$ and the interval $[0,1]$ when $\frac1{2}\leqslant a<1$. The limit $L$ of the sequence $\big(H^n(K_0)\big)_n$ depends on the starting set $K_0$. One has
\begin{gather*}
L=\cl\left(\bigcup_{(x_0,y_0)\in K_0} \big\{(x,x)\colon x\in h_{(x_0,y_0)}(\Gamma)\big\}\right),
\end{gather*}
where $h_{(x_0,y_0)}\colon {\mathbb R}\to {\mathbb R}$ is the affine map defined by $h_{(x_0,y_0)}(x)=\frac{x_0-y_0}{4} x+\frac{3y_0+x_0}{4}$. For example, when $a=\frac1{3}$, $L=\big({-}\frac1{4}\Gamma_{\frac1{3}} +\frac{7}{4},-\frac1{4} \Gamma_{\frac1{3}} +\frac{7}{4}\big)$ if $K_0=\{(1,2)\}$ (see Fig.~\ref{fig:limstoc}(a)) whereas $L=\big({-}\frac1{4} \Gamma_{\frac1{3}} +\frac{7}{4},-\frac1{4} \Gamma_{\frac1{3}} +\frac{7}{4}\big) \bigcup \big(\frac1{2} \Gamma_{\frac1{3}} +\frac1{2},\frac1{2} \Gamma_{\frac1{3}} +\frac1{2}\big)$ if $K_0=\{(1,2),(2,0)\}$ (see Fig.~\ref{fig:limstoc}(b)). More $K_0$ contains points then more complicated is the limit set $L$, with unions of overlapping Cantor sets.
\end{Example}

\begin{figure}[!h]\centering
\begin{minipage}{75mm}\centering
\includegraphics[scale=0.4]{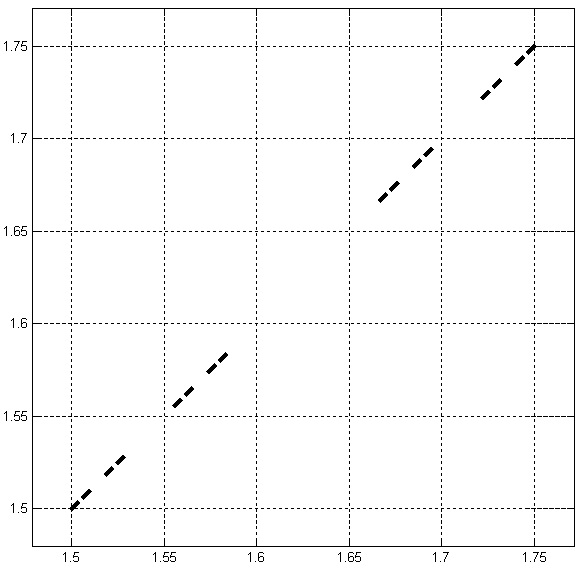}\\
{\footnotesize (a) one Cantor set}
\end{minipage}
\qquad
\begin{minipage}{75mm}\centering
\includegraphics[scale=0.4]{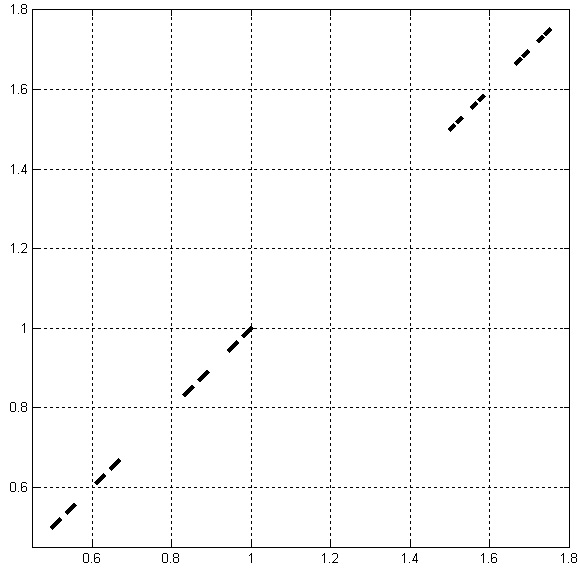}\\
{\footnotesize (b) two Cantor sets}
\end{minipage}
\caption{The limit set $L$ of $\big(H^n(K_0)\big)_n$ where $H$ is the Hutchinson operator associated with the IFS $\{f_1,f_2\}$. Maps $f_1$, $f_2$ are given in Example~\ref{exa:limstoc} with parameter $a=\frac1{3}$. Figure (a): the starting set is $K_0=\{(1,2)\}$ and $L$ is a Cantor set. Figure (b): the starting set is $K_0=\{(1,2),(2,1)\}$ and~$L$ is a union of two disjoint Cantor sets (one in bottom on the left and a~second in top on the right).}\label{fig:limstoc}
\end{figure}

Several necessary and sufficient conditions for a finite set of matrices to be a LCP set have been given (see \cite{beyn,bru,els1} and \cite{hart} for a survey). Not surprisingly, they require to evaluate the joint spectral radius of~${\cal M}$ or determine the generalized eigenspaces of the $A_i$'s.

\subsection{Identity-block matrices}\label{subsec:lincvgproj}

Hypothesis \eqref{eq:cvglcp} implies that the address function ${\cal A}\colon (i) \mapsto A_{(i)}$ is continuous. Unfortunately, very simple LCP sets may not fulfill this condition preventing from applying Lemma~\ref{lem:cvglcp}. This situation happens for example adding the matrix~$\id$ to a LCP set with a continuous function~${\cal A}$ (see~\cite{daub1}). However this simple case can be solved directly.

\begin{Lemma}\label{lem:limidlin} Assume that $A_p=\id$ and $\|A_i\|\leqslant 1$ for all $i\in\{1,\ldots,p-1\}$. Then, for all $K_0\in{\cal K}$, the sequence $\big(H^n(K_0)\big)_n$ converges to the set
\begin{gather*}\label{eq:limidlin}
L = \cl\left( \bigcup_{n\geqslant0} H^n(K_0)\right).
\end{gather*}
Moreover, if $\|A_i\|<1$ for all $i\in\{1,\ldots,p-1\}$, then denoting by $L'$ the attractor of the contractive $\mathrm{IFS}$ $\{f_1,\ldots,f_{p-1}\}$, we have $L=L'$ as soon as $K_0\subset L'$.
\end{Lemma}

\begin{proof} Since $f_p=\id$, the sequence $\big(H^n(K_0)\big)_n$ is clearly increasing. So it is enough to prove that it is bounded to get the convergence to $L$. Let $R>0$ be large enough to ensure $K_0\subset B(0,R)$. Then, $f_i(B(0,R))\subset B(0,R)$ for all $i\in\{1,\ldots,p\}$. Therefore $H^n(K_0)\subset B(0,R)$ for all $n\geqslant0$. Now, observe that, for all $K\in{\cal K}$, $H(K)=H'(K)\cup K$ where $H'$ is the Hutchinson operator associated with the $\mathrm{IFS}$ $\{f_1,\ldots,f_{p-1}\}$. Assume that $K_0\subset L'$. Then, $H(K_0)\subset H(L')=H'(L')\cup L' = L'\cup L' = L'$. By induction it follows that $H^n(K_0)\subset L'$ for all $n\geqslant 0$. Taking the limit we get $L\subset L'$. Next, we have $L=H(L)=H'(L) \cup L \supset H'(L)$. By induction it follows that $L\supset (H')^n(L)$ for all $n\geqslant 0$. Taking the limit we get $L\supset L'$, and finally $L=L'$.
\end{proof}

These latter sequences are related to the so-called inhomogeneous IFS's (see, e.g.,~\cite{fra}). Indeed, one has $H^n(K_0)=H_0^n(K_0)$ where $H_0$ is the Hutchinson operator associated with the contractive IFS $\{f_0,f_1,\ldots,f_{p-1}\}$ where $f_0\colon K\in{\cal K}\mapsto K_0$.

We can generalize the previous result to matrices $A_i$ which contain an identity-block, that is when the restriction of $f_i$ to a certain subspace of~${\mathbb R}^D$ is the identity function. We will state two results dealing with two special cases of such families.

\begin{Theorem}\label{theo:linproj} Assume that there exists two subspaces $V,W\subset {\mathbb R}^D$ satisfying $V\oplus W={\mathbb R}^D$ and, for all $i\in \{1,\ldots,p\}$,
\begin{enumerate}\itemsep=0pt
\item[$(i)$] $f_i(V)\subset V$ and the linear function $f_{i,V}\colon V\to V$ induced by $f_i$ is a contraction or the identity function $\id$,
\item[$(ii)$] $f_i(W)\subset W$ and the linear function $f_{i,W}\colon W\to W$ induced by $f_i$ is a contraction.
\end{enumerate}
Then, $\big(H^n(K_0)\big)_n$ converges for all $K_0\in {\cal K}$ to a set $L$. Precisely,
\begin{enumerate}\itemsep=0pt
\item[--] either there is at least one $i\in\{1,\ldots,p\}$ such that $f_{i,V}=\id$, then
\begin{gather*}
L=\cl\left(\bigcup_{n\geqslant 0} p_{V,W} \big(H^n(K_0)\big)\right),
\end{gather*}
where $p_{V,W}$ is the projection onto $V$ along $W$,
\item[--] or else $L=\{0\}$.
\end{enumerate}
Moreover, $L$ is the limit of the sequence $\big(H_V^n(p_{V,W}(K_0))\big)_n$ where $H_V$ is associated with the $\mathrm{IFS}$ $\{f_{1,V},\ldots,f_{p,V}\}$
\end{Theorem}

\begin{proof} Let us write $V_0= p_{V,W}(K_0)$, $W_0=p_{W,V}(K_0)$ where $p_{W,V}$ is the projection onto $W$ along $V$, and set $\widetilde{\lambda}_H=\max\limits_{1\leqslant i \leqslant p} \|f_{i,W}\|<1$. Finally let us write $K'_n = H_V^n(V_0)$.

(a) Let $x=v+w\in V\oplus W$. Since $f_i(x)=f_{i,V}(v)+f_{i,W}(w)$, we can write
\begin{gather*}
\forall\,n\geqslant0,\qquad H^n(K_0) = \bigcup_{v\in K'_n} v + L_n(v),
\end{gather*}
where $L_n(v) \subset W$ satisfies $p_{V,W}(L_n(v))=v$. Thus, we obtain
\begin{gather*}
\dha\big(H^n(K_0),K'_n\big) = \dha\left(\bigcup_{v\in K'_n} v + L_n(v),\bigcup_{v\in K'_n} \{v\} \right)
\leqslant \sup_{v\in K'_n} \dha(v + L_n(v), \{v\}).
\end{gather*}
Let $v\in K'_n$. We have
\begin{gather*}
\dha(v + L_n(v), \{v\}) = \sup_{z\in L_n(v)} \|z-v\| = \sup_{z\in L_n(v)} \|z- p_{V,W}(z)\|.
\end{gather*}
Let $z\in L_n(v)$. We can write
\begin{gather*}
z=(f_{i_n}\circ\cdots\circ f_{i_1})(z_0)=(f_{i_n,V}\circ\cdots\circ f_{i_1,V})(v_0) + (f_{i_n,W}\circ\cdots\circ f_{i_1,W})(w_0)
\end{gather*}
with $z_0=v_0+w_0 \in V_0\oplus W_0$. Thus,
\begin{gather*}
\|z- p_{V,W}(z)\|=\|(f_{i_n,W}\circ\cdots\circ f_{i_1,W})(w_0)\|\leqslant \|f_{i_n,W}\circ\cdots\circ f_{i_1,W}\|\|w_0\|\leqslant \widetilde{\lambda}_H^n \sup_{w_0\in W_0}\|w_0\|.
\end{gather*}
It follows that $\dha\big(H^n(K_0),K'_n\big) \leqslant \widetilde{\lambda}_H^n \sup\limits_{w_0\in W_0}\|w_0\| \to 0$ when $n\to\infty$.

(b) Therefore, it is now enough to prove the convergence of $(K'_n)_n$.

Assume first that $f_{i,V}\neq \id$ for all $i\in\{1,\ldots,p\}$. Then $H_V$ is contractive and we get $K'_n \to \{0\}$. Thus $K_n \to L=\{0\}$.

Assume now that there is at least one $j\in\{1,\ldots,p\}$ such that $f_{j,V}=\id$. Since $\|f_{i,V}\|\leqslant 1$, it follows from Lemma~\ref{lem:limidlin} that $K'_n\to L = \cl\big(\bigcup_{n\geqslant 0} K'_n\big)$. Let us prove by induction that $K'_n = p_{V,W}\big(H^n(K_0)\big)$. We have $K'_0=H_V^0(V_0)=V_0= p_{V,W}(K_0)$ and
\begin{align*}
K'_{n+1} & = H_V(K'_n)= \bigcup_{i=1}^p f_{i,V}(K'_n)= \bigcup_{i=1}^p f_{i,V}\big(p_{V,W}\big(H^n(K_0)\big)\big) \\
& = \bigcup_{i=1}^p p_{V,W}\big(f_i\big(H^n(K_0)\big)\big) = p_{V,W}\left(\bigcup_{i=1}^p f_i\big(H^n(K_0)\big)\right) = p_{V,W}\big(H^{n+1}(K_0)\big).
\end{align*}
Hence the result. Therefore $K_n \to L= \cl\big(\bigcup_{n\geqslant 0} p_{V,W}\big(H^n(K_0)\big)\big)$ that ends the proof.
\end{proof}

\begin{Theorem}\label{theo:linsup}Assume that there exists two subspaces $V,W\subset {\mathbb R}^D$ satisfying $V\oplus W={\mathbb R}^D$ and, for all $i\in \{1,\ldots,p\}$,
\begin{enumerate}\itemsep=0pt
\item[$(i)$] $(f_i-\id)(V) \subset W$,
\item[$(ii)$] $f_i(W)\subset W$ and the linear function $f_{i,W}\colon W\to W$ induced by $f_i$ is a contraction.
\end{enumerate}
Then, $\big(H^n(K_0)\big)_n$ converges for all $K_0\in {\cal K}$ to the set
\begin{gather*}
L=\cl\left(\bigcup_{v_0 \in p_{V,W}(K_0)} v_0 + \widetilde{L}(v_0)\right),
\end{gather*}
where $p_{V,W}$ is the projection onto $V$ along $W$ and $\widetilde{L}(v_0)$ is the attractor of the $\mathrm{IFS}$ $\big\{\widetilde{f}_1,\ldots,\widetilde{f}_p\big\}$ where $\widetilde{f}_i\colon W\to W$ is defined by $\widetilde{f}_i(w)= f_{i,W}(w) + (f_i-\id)(v_0)$.

In particular, $p_{V,W}(L)=p_{V,W}(K_0)$.
\end{Theorem}

\begin{proof} Let us write $V_0= p_{V,W}(K_0)$, $W_0=p_{W,V}(K_0)$ where $p_{W,V}$ is the projection onto $W$ along $V$, and set $\widetilde{\lambda}_H=\max\limits_{1\leqslant i \leqslant p} \|f_{i,W}\|=\max\limits_{1\leqslant i \leqslant p} \|\widetilde{f}_i\|<1$.

(a) First let $z_0\in K_0$ with $z_0 = v_0 + w_0 \in V_0+W_0 \subset V\oplus W$. We have
\begin{gather*}
\forall\,i\in\{1,\ldots,p\},\qquad f_i(z_0) = v_0 + \big((f_i(v_0)-v_0) + f_i(w_0)\big) = v_0 + \widetilde{f}_i(w_0) \in V_0+W.
\end{gather*}
We then prove by induction that
\begin{gather*}
\forall\,n\geqslant0,\qquad H^n(\{z_0\}) = v_0 + \widetilde{H}^n(\{w_0\}),
\end{gather*}
where $\widetilde{H}$ is the Hutchinson operator associated with the contractive IFS $\big\{\widetilde{f}_1,\ldots,\widetilde{f}_p\big\}$. Therefore the sequence $\big(\widetilde{H}^n(w_0)\big)_n$ converges to a set $\widetilde{L}(v_0)$ which only depends on $v_0$. It follows that $\big(H^n(\{z_0\})\big)_n$ converges to $v_0 + \widetilde{L}(v_0)$ with the estimate
\begin{gather*}
\dha\big(v_0 + \widetilde{L}(v_0),v_0+\widetilde{H}^n(w_0)\big) = \dha\big(\widetilde{L}(v_0),\widetilde{H}^n(w_0)\big) \leqslant \widetilde{\lambda}_H^n\dha\big(\widetilde{L}(v_0),w_0\big).
\end{gather*}

(b) Let
\begin{gather*}
R=\frac1{1-\widetilde{\lambda}_H} \max_{1\leqslant i\leqslant p} \big\{ \rho((f_i-\id)(V_0)) \big\}.
\end{gather*}
We have $\widetilde{f}_i(B(0,R)\cap W)\subset B(0,R)\cap W$ for all $i\in\{1,\ldots,p\}$,
so that $\widetilde{H}(B(0,R)\cap W)\subset B(0,R)\cap W$. It follows that $\widetilde{L}(v_0) \subset B(0,R)\cap W$ for all $v_0\in V_0$. Thus
\begin{gather*}
\forall\,v_0\in V_0,\qquad \rho\big(v_0 + \widetilde{L}(v_0)\big) \leqslant \|v_0\| + R \leqslant \rho(V_0)+R.
\end{gather*}
We then have proved that the set $\,\bigcup_{v_0 \in V_0} v_0 + \widetilde{L}(v_0)$ is bounded, i.e., $L\in{\cal K}$.

Furthermore, since $\widetilde{L}(v_0)$ does not depend on $w_0$, we can write
\begin{gather*}
\bigcup_{v_0 \in V_0} v_0 + \widetilde{L}(v_0) = \bigcup_{v_0+w_0 \in K_0} v_0 + \widetilde{L}(v_0).
\end{gather*}

(c) Finally, writing $H^n(K_0)=\bigcup_{v_0+w_0 \in K_0} v_0 + \widetilde{H}^n(w_0) $ and using (a) and (b), we obtain
\begin{align*}
\dha\big(L,H^n(K_0)\big) & = \dha\left(\bigcup_{v_0+w_0 \in K_0} v_0 + \widetilde{L}(v_0),\bigcup_{v_0+w_0 \in K_0} v_0 + \widetilde{H}^n(w_0)\right) \\
 & \leqslant \sup_{v_0+w_0\in K_0} \dha\big(v_0 + \widetilde{L}(v_0),v_0+\widetilde{H}^n(w_0)\big) \\
 & \leqslant \widetilde{\lambda}_H^n\sup_{v_0+w_0\in K_0}\dha\big(\widetilde{L}(v_0),w_0\big).
\end{align*}

Since this latter supremum is finite, we obtain the result by letting $n$ goes to $\infty$.
\end{proof}

Notice that hypotheses of Theorem \ref{theo:linsup} mean that, for all $i\in\{1,\ldots,p\}$, the maps $f_i$'s have a block matrix with respect to the sum $V\oplus W$ of the form
\begin{gather*} \left[\begin{matrix}
\id & 0 \\
M_i & \widetilde{A}_i
\end{matrix}\right] \end{gather*}
with $\widetilde{A}_i$ contractive and some matrix $M_i$. We deduce that $\sigma_{{\cal M}}=1$.

In particular, when $M_i=0$ for all $i\in\{1,\ldots,p\}$, we get $f_{i,W}=\widetilde{f}_i$ hence $\widetilde{L}(v_0)=\{0\}$ for all $v_0\in V_0$ and the limit set is simply $L=p_{V,W}(K_0)$. We recover a special case of Theorem \ref{theo:linproj}, namely when $f_{i,V}=\id$ and $H_V=\id$.

\begin{Example}\label{exa:linsup}
Let us consider the IFS $\{f_1,f_2\}$ where the $f_i\colon {\mathbb R}^3\to {\mathbb R}^3$ are the linear maps given by their canonical matrices
\begin{gather*} A_1 = \left[ \begin{matrix}
1 & 0 & 0 \\
0 & 1 & 0 \\
a & b & c
\end{matrix}\right]
\qquad \text{and}\qquad A_2 = \left[ \begin{matrix}
1 & 0 & 0 \\
0 & 1 & 0 \\
0 & 0 & c
\end{matrix}\right] \end{gather*}
with $a,b\in{\mathbb R}$ and $0<c<1$. Then, for all $K_0\in{\cal K}$, the sequence $\big(H^n(K_0)\big)_n$ converges to the set
\begin{gather*}
L=\cl\left(\bigcup_{(x_0,y_0,z_0)\in K_0} \left\{\left(x_0,y_0,\frac{ax_0+by_0}{1-c} L_c\right)\right\}\right)
\end{gather*}
where $L_c$ is the Cantor set $\Gamma_c$ if $0<c<\frac1{2}$ and the interval $[0,1]$ if $\frac1{2}\leqslant c<1$.

As an example, $L$ is shown in Fig.~\ref{fig:limlinsup} when $K_0$ is the unit circle $\{(\cos t,\sin t,0)\colon t\in[0,2\pi]\}$ and parameters $(a,b,c)=\big(1,1,\frac1{4}\big)$.
\end{Example}

\begin{figure}[!h]\centering
\includegraphics[scale=0.5]{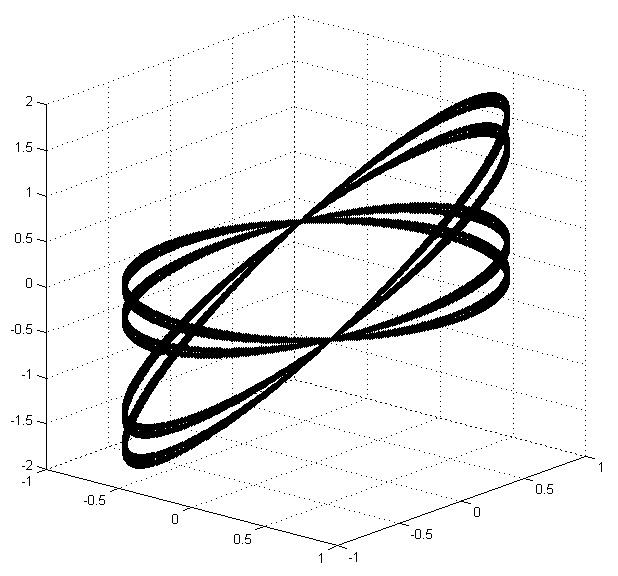}
\caption{The limit set $L$ of $\big(H^n(K_0)\big)_n$ where $H$ is the Hutchinson operator associated with the IFS $\{f_1,f_2\}$. Maps $f_1$, $f_2$ are given in Example~\ref{exa:linsup} with parameters $(a,b,c)=\big(1,1,\frac1{4}\big)$. The starting set is the circle $K_0=\{(\cos t,\sin t,0) \colon t\in[0,2\pi]$\}.}\label{fig:limlinsup}
\end{figure}

\begin{proof}
We can apply Theorem \ref{theo:linsup} with $V=\Span\{(1,0,0),(0,1,0)\}$ and $W=\Span\{(0,0,1)\}$. Thus, $\big(H^n(K_0)\big)_n$ converges to the set
\begin{gather*} L=\cl\left(\bigcup_{(x_0,y_0,z_0)\in K_0} (x_0,y_0,0)+L(x_0,y_0)\right), \end{gather*}
where $L(x_0,y_0)$ is the attractor of the IFS $\big\{\widetilde{f}_1,\widetilde{f}_2\big\}$ with, for all $w=(0,0,z)\in W$, \begin{gather*} \widetilde{f}_1(w)=\widetilde{f}_1(0,0,z)= (0,0,cz+ax_0+by_0 )\qquad \text{and}\qquad \widetilde{f}_2(w)=\widetilde{f}_2(0,0,z)= (0,0,cz ). \end{gather*}
By uniqueness, we check that this attractor is the one announced.

Assume that $(a,b,c)=\big(1,1,\frac1{4}\big)$ and $K_0=\{(\cos t,\sin t)\colon t\in[0,2\pi]\}\times\{0\}$. We have
\begin{gather*}
L= \cl\left(\bigcup_{t\in[0,2\pi]} \big(\cos t,\sin t,\tfrac{4}{3}(\cos t+\sin t)\Gamma_{\frac1{4}}\big)\right).
\end{gather*}
Then $L$ is the closure of a union of circles drawn on the cylinder $\{(\cos t,\sin t)\colon t\in[0,2\pi]\}\times \big\{\frac{4\sqrt{2}}{3}t\colon t\in[-1,1]\big\}$. Each intersection with a generatrice is homothetic with the Cantor set $\Gamma_{\frac1{4}}$ (except for the two special values $t=\frac{3\pi}{4}$ and $t=\frac{7\pi}{4}$).
\end{proof}

\subsection{Orbit of the unit ball}\label{subsec:linunitball}

To avoid the various behaviors due to the different invariant subspaces of the $A_i$'s, we propose to take into account all the directions of ${\mathbb R}^D$ by focusing on the $H$-orbit of the unit ball $B(0,1)$.

\begin{Proposition}\label{prop:linboule}
Assume that
\begin{enumerate}\itemsep=0pt
\item[$(i)$] For all $i\in\{1,\ldots,p\}$, $\,\|A_i\|\leqslant 1$,
\item[$(ii)$] There exists $N\geqslant1$ indices $i_1,\ldots,i_N\in\{1,\ldots,p\}$ such that the matrix $A_{i_N}\cdots A_{i_1}$ has eigenvalue~$1$.
\end{enumerate}
Then, the sequence $(H^n(B(0,1)))_n$ is decreasing and converges to the set
\begin{gather}\label{eq:limlinboule}
L = \bigcap_{n\geqslant0} H^n(B(0,1))
\end{gather}
with $\rho(L)=1$. Moreover, $H_{\rho}^n(B(0,1))=H^n(B(0,1))$ and $d_n=1$ for all $n\geqslant 0$.
\end{Proposition}

\begin{proof}Hypothesis (i) implies that $H(B(0,1))\subset B(0,1)$, then $\big(H^n(B(0,1))\big)_n$ is decreasing and converges to $L$. Let $v\in B(0,1)$ with $\|v\|=1$ and such that $A_{i_N}\cdots A_{i_1}v=v$. One has $v\in H^{k N}(K_0)$ for all $k\geqslant0$, thus $v\in L$. It follows $L\neq\{0\}$ and $v\in H^n(B(0,1))$ for all $n\geqslant0$. Therefore,
\begin{gather*}
1=\|v\|\leqslant \rho(L)\leqslant \rho\big(H^n(B(0,1))\big) \leqslant \rho(B(0,1))=1.
\end{gather*}
This yields $H_{\rho}^n(B(0,1))=H^n(B(0,1))$ and $d_n=1$.
\end{proof}

Notice that the two hypotheses imply that $\alpha(A_{i_N}\cdots A_{i_1})=\|A_{i_N}\cdots A_{i_1}\|=1$, hence $\sigma_{{\cal M}}=1$. Notice also that~${\cal M}$ need not to be a~LCP set. One can for example consider matrices with rotations or symmetries.

\begin{Example}\label{exa:limlinboule1}Let us consider the IFS $\{f_1,f_2\}$ where the $f_i\colon {\mathbb R}^2\to {\mathbb R}^2$ are the linear maps given by their canonical matrices
\begin{gather*} A_1 = \left[\begin{matrix}
a & 1 \\
0 & a
\end{matrix}\right] \qquad \text{and}\qquad
A_2 = \left[ \begin{matrix}
a & 0 \\
1 & a
\end{matrix}\right] \end{gather*}
with $a\geqslant 0$. Then, the sequence $\big(H_{\rho}^n(B(0,1))\big)_n$ is decreasing, thus converges to a~set~$L$. Moreover, the point
\begin{gather*}
v=\left(\frac1{\sqrt2}\frac{1+\sqrt{1+4a^2}}{\sqrt{1+4a^2+\sqrt{1+4a^2}}},\frac1{\sqrt2}\frac{2a}{\sqrt{1+4a^2+\sqrt{1+4a^2}}}\right)
\end{gather*}
belongs to $L$ and satisfies $\|v\|=1$. As an example, $L$ is shown in Fig.~\ref{fig:limlinboule1} when $a=1$.
\end{Example}

\begin{figure}[!h]\centering
\includegraphics[scale=0.3]{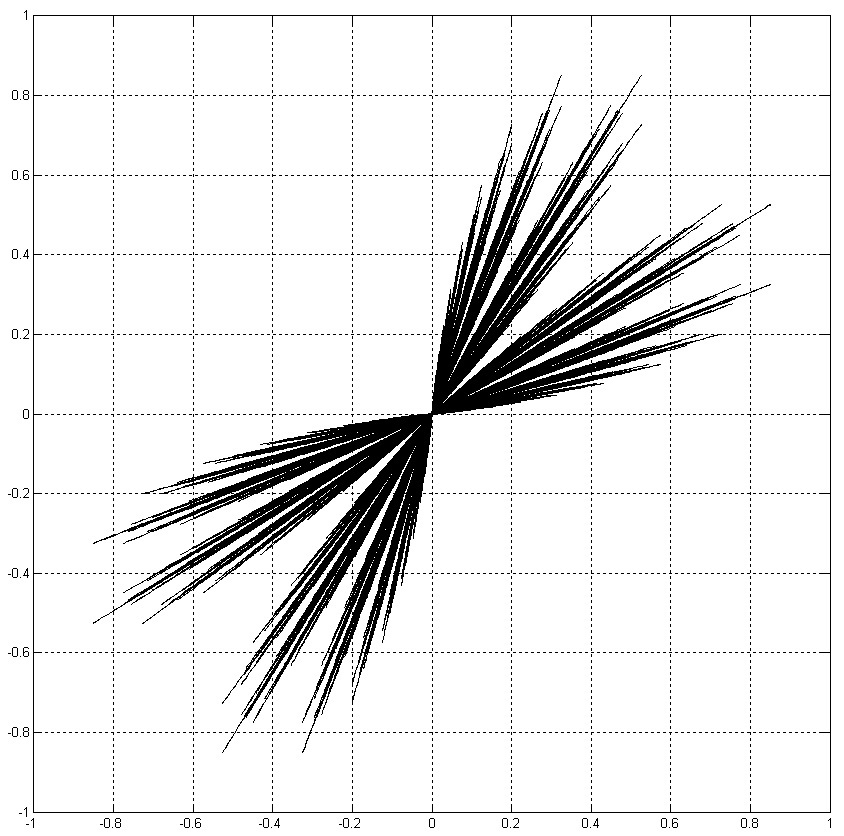}
\caption{The limit set $L$ of $(H_{\rho}^n(B(0,1)))_n$ where $H$ is the Hutchinson operator associated with the IFS $\{f_1,f_2\}$. Maps $f_1$, $f_2$ are given in Example~\ref{exa:limlinboule1} with parameter $a=1$.}\label{fig:limlinboule1}
\end{figure}

\begin{proof} Since $\|A_1\|^2=\|A_2\|^2=a^2+\frac1{2}\big(1+\sqrt{1+4a^2}\big)\geqslant 1$, we cannot apply directly the previous proposition. Thus we have to consider the normalized matrices $A'_i=\frac1{d}A_i$ with $d=\|A_1\|=\|A_2\|$. Notice that we can show here that $d=\sigma_{{\cal M}}$ (see \cite{theys}). Then, hypothesis (i) is satisfied and one checks that the matrix
\begin{gather*}
A'_1 A'_2 = \frac1{a^2+\frac1{2}\big(1+\sqrt{1+4a^2}\big)}\left[\begin{matrix}
1+a^2 & a \\
a & a^2
\end{matrix}\right]
\end{gather*}
has eigenvalue $1$. Hence, using now Proposition~\ref{prop:linboule}, the sequence $\big(H_d^n(B(0,1))\big)_n$ converges to the set
\begin{gather*}
L = \bigcap_{n\geqslant0} H_d^n(B(0,1))
\end{gather*}
with $\rho(L)=1$. Therefore $\big(H_{\rho}^n(B(0,1))\big)_n$ converges to $L$ by Proposition~\ref{prop:linred}. Finally, it follows from the proof of Proposition~\ref{prop:linboule} that~$v$, one of the unit eigenvectors associated with~$1$, belongs to~$L$.
\end{proof}

Let us emphasize that the previous method may be applied generally when two matrices of~${\cal M}$ are symmetric each other, or when one of them is symmetric. In particular we can use it to study the example of Section~\ref{subsec:motivation}.

\subsection[Angular structure of the limit set in ${\mathbb R}^2$]{Angular structure of the limit set in $\boldsymbol{{\mathbb R}^2}$}\label{subsec:polarlim}

Under the hypotheses of Proposition~\ref{prop:linboule}, the set $L$ defined by \eqref{eq:limlinboule} is an eigenset for $H$ associated with the eigenvalue $\sigma_{{\cal M}}=1$. It is straightforward to see that, as the ball $B(0,1)$, the set $L$ is symmetric and star-shaped with respect to the origin, i.e., if $x\in L$ then $rx \in L$ for any $r\in[-1,1]$. Such a property was also discussed in~\cite{barvin11}. More generally, if $K_0$ is symmetric and star-shaped with respect to the origin, then the same holds for every set $H^n(K_0)$ and thus for its limit set $L$ in case of convergence. Therefore, such a limit set $L$ admits a polar representation and it would be especially interesting to know its angular structure. In this section we state a~result in this direction for IFS's in~${\mathbb R}^2$.

Let ${\mathbb P} =\big\{(x,y)\in{\mathbb R}^2 \colon x\geqslant0\big\}\backslash\{0\}$. Every point $(x,y)\in{\mathbb P}$ may be written with polar coordinates as $(R\cos\theta,R\sin\theta)$ with $R>0$ and $\theta\in\big[{-}\tfrac{\pi}{2},\tfrac{\pi}{2}\big]$. Let $K\in\cal{K}$ with $K\neq\{0\}$. Assume moreover that $K$ is symmetric with respect to the origin. Then $K\cap {\mathbb P} \neq\varnothing$ and we can define its \textit{set of slopes} $S_K\subset[-\infty,\infty]$ by
\begin{gather*}\label{eq:defslope}
S_K=\big\{\tan\theta \,|\, \exists\,(R\cos\theta,R\sin\theta)\in K\cap {\mathbb P} \big\}
\end{gather*}
with the convention $\tan\big({\pm}\frac{\pi}{2}\big)=\pm\infty$.

The following result provides a description of the set of slopes of $L$ for particular IFS's.

\begin{Proposition}\label{prop:homo}
Let $p\geqslant1$ linear maps $f_i\colon {\mathbb R}^2 \to {\mathbb R}^2$ given by their canonical matrices
$ A_i = \left[ \begin{smallmatrix}
a_i & b_i \\
c_i & d_i
\end{smallmatrix}\right]$
such that $\det(A_i)\neq 0$ and $f_i({\mathbb P})={\mathbb P}$. Let $K_0\subset{\mathbb R}^2$ be a symmetric and star-shaped with respect to the origin set. Assume that the sequence $\big(H^n(K_0)\big)_n$ converges to a set $L\neq\{0\}$. Then $L$ is symmetric and star-shaped with respect to the origin and its set of slopes $S_L$ is a non-empty invariant set of the Hutchinson operator $\widehat{H}=\bigcup_{i=1}^p \widehat{f}_i$ where $\widehat{f}_i\colon [-\infty,\infty]\to[-\infty,\infty]$ is the homographic function defined by
\begin{gather*}\label{def:homo}
\widehat{f}_i(z)= \frac{d_i z +c_i}{b_i z + a_i}.
\end{gather*}
\end{Proposition}

\begin{proof} Notice that $S_L$ is well-defined. Writing $L\cap{\mathbb P}=\{(R\cos\theta,R\sin\theta) \colon R>0, \theta\in\Theta\}$ we get $S_L=\tan\Theta$. Observing that $L \cap {\mathbb P} = H(L)\cap {\mathbb P}$, we have $S_L=S_{H(L)}$. We deduce now from the hypothesis that
\begin{gather*}
H(L) \cap {\mathbb P} =\bigcup_{i=1}^p \big( f_i(L) \cap {\mathbb P} \big) = \bigcup_{i=1}^p \big( f_i(L) \cap f_i({\mathbb P}) \big)
= \bigcup_{i=1}^p f_i(L\cap {\mathbb P}) = H(L\cap {\mathbb P}).
\end{gather*}
Hence $H(L) \cap {\mathbb P}$ is the set of the points $f_i(s)$ for all $s\in L\cap {\mathbb P}$ and $i\in\{1,\ldots,p\}$. Since
\begin{gather*}
f_i(s)=A_i \left[\begin{matrix}
 R\cos \theta \\
 R\sin \theta
 \end{matrix}\right] = R\left[\begin{matrix}
 a_i\cos \theta + b_i\sin \theta \\
 c_i\cos \theta + d_i\sin \theta
 \end{matrix}\right],
\end{gather*}
the set of slopes of $H(L)$ is exactly the set of all the points of the form
\begin{gather*}
\frac{c_i\cos \theta + d_i\sin \theta}{a_i\cos \theta + b_i\sin \theta} = \widehat{f}_i(\tan \theta),\qquad \theta\in\Theta\qquad \text{and}\qquad i\in\{1,\ldots,p\}.
\end{gather*}
Therefore, $S_{H(L)}=\bigcup_{i=1}^p \widehat{f}_i(\tan\Theta)$. The results follows.
\end{proof}

The angular structure of~$L$ is then known as soon as we can describe the invariant sets of the Hutchinson operator~$\widehat{H}$. This operator may have several invariant sets~$S$, not necessarily closed. However, there is a useful way to determine them. Indeed, if $\widehat{H}$ is contractive then every bounded invariant set~$S$ of $\widehat{H}$ satisfies $\cl(S)=\widehat{L}$ where $\widehat{L}$ is the attractor of $\widehat{H}$. It is possible to determine such attractors of IFS made up with homographic functions (see for example~\cite[p.~136]{falco}). Notice that $[-\infty,\infty]$ is always a compact set and then would be an invariant set. Thus, it will be often necessary to consider a~restriction of $\widehat{H}$ to obtain a contractive operator.

\begin{Example}\label{exa:limcantor}Let us consider the IFS $\{f_1,f_2\}$ where the $f_i\colon {\mathbb R}^2\to {\mathbb R}^2$ are the linear maps given by their canonical matrices
\begin{gather*} A_1 = \left[\begin{matrix}
1 & 0 \\
0 & a
\end{matrix}\right] \qquad \text{and}\qquad A_2 = \left[ \begin{matrix}
\dfrac{b}{1-a} & 0 \\
b & \dfrac{a\, b}{1-a}
\end{matrix}\right] \end{gather*}
with $0<a<1$, $0<b<1$ and $a+b\leqslant1$. Starting from the unit square $K_0=[-1,1]\times[-1,1]$, the sequence $\big(H^n(K_0)\big)_n$ converges to the set $L=\bigcap_{n\geqslant0}H^n(K_0)$. Its set of slopes satisfies $\cl(S_L)=\Gamma_a$ if $0<a<\frac1{2}$ and $\cl(S_L)=[0,1]$ if $\frac1 2 \leqslant a\leqslant 1$ (see Fig.~\ref{fig:limcantor}(a)). Moreover, if $b=1-a$ then $S_L=\cl(S_L)$ (see Fig.~\ref{fig:limcantor}(b)).
\end{Example}

\begin{figure}[!h]\centering
\begin{minipage}{75mm}\centering
\includegraphics[scale=0.3]{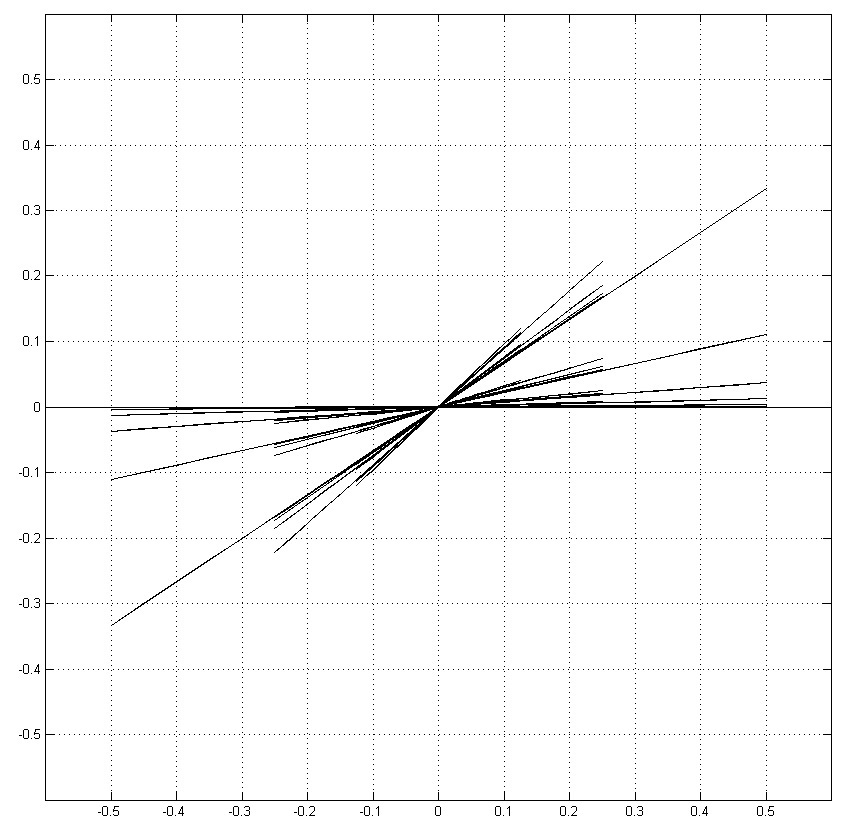}\\
{\footnotesize (a) $\cl(S_L)=\Gamma_{\frac1{3}}$}
\end{minipage}\qquad
\begin{minipage}{75mm}\centering
\includegraphics[scale=0.3]{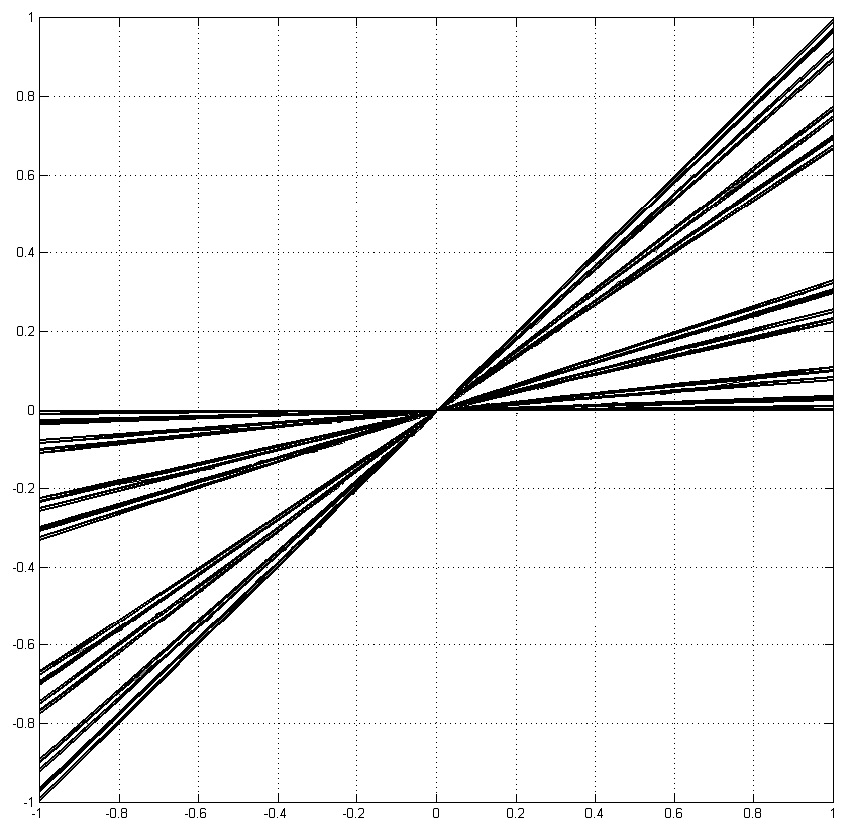}\\
{\footnotesize (b) $S_L=\Gamma_{\frac1{3}}$}
\end{minipage}
\caption{The limit set $L$ of $\big(H^n(K_0)\big)_n$ where $H$ is the Hutchinson operator associated with the IFS $\{f_1,f_2\}$. The starting set is the unit square $[-1,1]\times[-1,1]$ and maps $f_1$, $f_2$ are given in Example~\ref{exa:limcantor}. In (a) parameters are $a=b=\frac1{3}$ and the $S_L$ is dense in the triadic Cantor set. In (b) parameters are $a=\frac1{3}$, $b=\frac{2}{3}$ and $S_L$ is the triadic Cantor set.}\label{fig:limcantor}
\end{figure}

\begin{proof} One checks that $H(K_0)\subset K_0$ so that the sequence $\big(H^n(K_0)\big)_n$ converges to the given set $L$. Moreover all the hypothesis on the $A_i's$ of Proposition~\ref{prop:homo} are satisfied. For example $(1,0)\in H^n(K_0)$ for all $n\geqslant0$ so that $L\neq\{0\}$. Hence the set of slopes $S_L$ is an invariant set of the operator $\widehat{H}=\widehat{f}_1\cup\widehat{f}_2$ where $\widehat{f}_1(z)=az$ and $\widehat{f}_2(z)=az+(1-a)$. Since $\|f_1(0,y)\|<1$ and $\|f_2(0,y)\|<1$ for all $y\neq0$, one has $\Span\{(0,1)\}\cap L = \{0\}$. It follows that $S_L\subset{\mathbb R}$. Since $\widehat{H}$ is contractive on ${\cal K}_{{\mathbb R}}$ one obtains $\cl(S)=\Gamma_a$ if $0<a<\frac1{2}$ and $\cl(S)=[0,1]$ if $\frac1{2}\leqslant a <1$.

Assume now that $b=1-a$. To obtain a more precise description of the limit set $L$ we rather apply Theorem~\ref{theo:linsup} with $V=\Span\{(1,0)\}$ and $W=\Span\{(0,1)\}$. Let $z_0=(x_0,y_0) \in K_0$. We have $p_{V,W}(z_0)=(x_0,0)$. It follows that $\big(H^n(K_0)\big)_n$ converges to the set $L=\cl\big(\bigcup_{x_0 \in [-1,1]} (x_0,0)+ L(x_0)\big)$ where $L(x_0)$ is the attractor of the IFS $\big\{\widetilde{f}_1,\widetilde{f}_2\big\}$ with, for all $w=(0,y)\in W$,
\begin{gather*}
\widetilde{f}_1(w)=\widetilde{f}_1(0,y)=(0,ay)\qquad \text{and }\qquad \widetilde{f}_2(w)=\widetilde{f}_2(0,y)=(0,ay+(1-a)x_0).
\end{gather*}
By uniqueness, this attractor is $L(x_0) = (0,x_0 L_a)$ with $L_a=\Gamma_a$ if $a<\frac1{2}$ and $L_a=[0,1]$ if $a\geqslant\frac1{2}$. The limit set is then{\samepage
\begin{gather*}
L=\cl\left(\bigcup_{x_0 \in [-1,1]} (x_0,x_0 L_a) \right)
\end{gather*}
and we directly see that $S_L=L_a$.}
\end{proof}

The fact that in Example~\ref{exa:limcantor} the set $S_L$ is not always the whole Cantor set $\Gamma_a$ comes from the function~$f_2$. When $b\neq 1-a$, $f_2$ is a contraction. Thus, all the orbits of points $z\in B(0,1)$ associated with~$f_2$ infinitely many times correspond to a~`slope' $s=\tan \theta$ for which $R=R(\theta)=0$. These slopes are then not visible in the limit set.

\section{Other renormalizations}\label{sec:general}

We have renormalized the sets $H^n(K_0)$ by dividing them by their radius. As we mentioned in the introduction, one usually rather uses the diameter
to rescale a sequence of compact sets. Thus, we can wonder what a such renormalization would yield. More generally, we want to study in this section the iteration of more general operators $H_{\fhi}$ which will provide various ways to approximate the solutions of the eigen-equation $H(K)=dK$. We keep using the sequence~$(K_n)_n$ but choosing a well-adapted operator according to the form of the matrices $A_i$'s.

\subsection{Renormalization with a size function}\label{subsec:renormsize}

We are interested in functions $\fhi$ that describe the size of a compact set and the way it occupies the space. Following the example of the radius function $\rho$, we will say that a function $\fhi\colon {\cal K} \rightarrow [0,+\infty)$ is a \textit{size function} if it is continuous with respect to $\dha$, monotonic and homogeneous (see Section~\ref{subsec:strategy}). For example, the max-radius function $\rho_{\infty}$ and the diameter $\delta$ respectively defined on ${\cal K}$ by
\begin{gather*}\label{def:maxradiusdiam}
\rho_{\infty}(K)= \max_{(x_1,\ldots,x_D)\in K}\{|x_j|\colon 1\leqslant j \leqslant D\}\qquad\text{and}\qquad \delta(K)= \max\{\|x-y\|\colon x,y\in K\}
\end{gather*}
are two size functions.

We define an associated operator $H_{\fhi}$ in the same way as \eqref{def:normhutch} setting
\begin{gather*}
\forall\,K\in{\cal K},\qquad H_{\fhi}(K)= \frac1{\fhi(H(K))} H(K)
\end{gather*}
and consider the $H_{\fhi}$-orbit of some set $K_0\in{\cal K}$. Let us keep all the notation of the previous sections, easily adapted by replacing $\rho$ with~$\fhi$. In particular $K_n=H_{\fhi}^n(K_0)$ (see~\eqref{def:normhutchiter}) and $d_n=\fhi\big(\bigcup_{i=1}^p f_i(K_n)\big)$ (see the first equality in~\eqref{eq:defdn}). We will assume that $K_n$ is always well-defined, i.e., $\fhi(K_n)\neq 0$.

We are still interested in the convergence of the sequence $(K_n)_n$ and the description of its limit. The key-points are Theorem~\ref{theo:renorm} and Lemma~\ref{prop:linred} which provide general conditions of convergence. We summarize here the main results which still hold for any size functions.

\begin{Theorem}\label{theo:cvgfhi}Let $p\geqslant1$ affine maps $f_i\colon {\mathbb R}^D\rightarrow {\mathbb R}^D$ of the form $f_i(x)=A_i x + b_i$.
\begin{enumerate}\itemsep=0pt
\item[$(i)$] If $(d_n)_n$ converges to $d>\lambda_H$, then $(K_n)_n$ converges to the attractor $L_d$ and $d$ satisfies the inequality
\begin{gather*}\label{eq:ineqfhi}
\fhi\big(\big\{(d\id -A_1)^{-1}b_1,\ldots,(d\id -A_p)^{-1}b_p\big\}\big)\leqslant 1.
\end{gather*}
\item[$(ii)$] Assume that $b_i=0$ for all $i\in\{1,\ldots,p\}$. If $\big(H^n(K_0)\big)_n$ converges to a set $L\in{\cal K}$ such that $\fhi(L)\neq 0$, then $(K_n)_n$ converges to $K=\frac1{\fhi(L)}L$.
\end{enumerate}
\end{Theorem}

In particular, when all the $f_i$'s are linear, the choice of $\fhi$ is not important. Therefore, all the results of Section~\ref{sec:linear} providing the convergence of the sequence $\big(H^n(K_0)\big)_n$ (e.g., Lemma~\ref{lem:cvglcp}, Lemma~\ref{lem:limidlin}, Theorem~\ref{theo:linsup}) and the description of its limit $L$ (Proposition~\ref{prop:homo}) may be used. In case of convergence, the limit set $K$ will only depend on $\fhi$ through the scaling factor~$\fhi(L)$.

\subsection{Renormalization with the max-radius function}\label{subsec:maxradius}

We consider here the size function $\rho_{\infty}$, which is nothing but the `radius function' associated with the usual maximum norm $\|\cdot\|_{\infty}$. In particular, Property~\eqref{eq:rhostab}, Lemma~\ref{lem:bounddn} and Proposition~\ref{prop:majd} still work \textit{mutatis mutandis}. Therefore it should be possible to obtain a result similar to Theorem~\ref{theo:radius}. Since the Euclidean norm is isotropic we assumed in Theorem~\ref{theo:radius} the $A_i$'s were homotheties. Here, the maximum norm allows us to deal with more general diagonal matrices. However, many complicated particular situations may happen. For the sake of simplicity, we only state a simpler result with an additional hypothesis avoiding these special behaviors.

\begin{Proposition}\label{theo:maxradius} Let $p\geqslant1$ maps $f_i\colon {\mathbb R}^D\rightarrow {\mathbb R}^D$ of the form $f_i(x)=A_i x+b_i$ where $A_i=\diag(a_{i,1},\ldots,a_{i,D})$ is a diagonal matrix with non-negative entries and $b_i=(b_{i,1},\ldots,b_{i,D})$. Let $K_0\in{\cal K}$. Assume that
\begin{gather}\label{eq:theomaxradius}
\max_{\substack{1\leqslant i \leqslant p \\ 1\leqslant j \leqslant D}}\left\{|a_{i,j}x_j+ b_{i,j}|\colon x=(x_1,\ldots,x_D)\in H_{\rho_{\infty}}(K_0)\right\}
> \max_{\substack{1\leqslant i \leqslant p \\ 1\leqslant j \leqslant D}}\big\{|a_{i,j}|\big\}.
\end{gather}
Then, the sequence $\big(H_{\rho_{\infty}}^n(K_0)\big)_n$ converges to $L_d$ where $d=\lim\limits_{n\to\infty} d_n$.
\end{Proposition}

\begin{proof} The proof is very similar to that of Theorem \ref{theo:radius}. First we claim that $(d_n)_{n\geqslant 1}$ always converges. Indeed, for all $n\geqslant 1$ we can find $x_n\in K_n$, $i_n\in\{1,\ldots,p\}$ and $j_n\in\{1,\ldots,D\}$ such that $d_n=|a_{i_n,j_n} x_{n,j_n} + b_{i_n,j_n}|$. Then, $u_n=\frac1{d_n}(A_{i_n} x_n + b_{i_n})$ satisfies $u_n\in K_{n+1}$ and $|u_{n,j_n}|=1$. Since $|x_{n,j_n}|\leqslant 1$ we obtain
\begin{align*}
d_{n+1}\geqslant \max_{1\leqslant j \leqslant D}\{|a_{i_n,j} u_{n,j} + b_{i_n,j}|\}
& \geqslant \max_{1\leqslant j \leqslant D} \{|a_{i_n,j} u_{n,j} + d_n u_{n,j}| - |d_n u_{n,j} - b_{i_n,j}|\} \\
& \geqslant |a_{i_n,j_n} u_{n,k_n} + d_n u_{n,j_n}| - |d_n u_{n,j_n} - b_{i_n,j_n}| \\
& \geqslant (a_{i_n,j_n} + d_n)|u_{n,j_n}| - a_{i_n,j_n}|x_{n,j_n}| \\
& \geqslant d_n + a_{i_n,j_n}(1-|x_{n,j_n}|) \geqslant d_n.
\end{align*}
Hence, $(d_n)_{n\geqslant1}$ is increasing and bounded so it converges to a number $d\geqslant0$. Notice that the left-hand side of \eqref{eq:theomaxradius} is $d_1$ and the right-hand side of \eqref{eq:theomaxradius} is $\lambda_H$. Thus one has $d\geqslant d_1 >\lambda_H$ and the result follows from Theorem \ref{theo:renorm}.
\end{proof}

\subsection{Renormalization with the diameter function}\label{subsec:diam}

We consider now the diameter function $\delta$. The situation is more complicated than the previous ones, even if the matrices $A_i$'s are homotheties. The stability property \eqref{eq:rhostab} of $\rho$ was a key-point in the proof of Theorem \ref{theo:radius} but unfortunately it is no longer satisfied. We will only deal with the one dimensional case $D=1$.

\begin{Proposition}\label{prop:diamD1}Let $p\geqslant1$ affine maps $f_i\colon {\mathbb R}\rightarrow {\mathbb R}$ of the form $f_i(x)=\alpha_i x + b_i$ with $\alpha_i\geqslant 0$. Let us define the function
\begin{gather}\label{eq:F}
F(x)= \frac{\min\limits_{1\leqslant i\leqslant p} \{\alpha_ix+b_i\}}{\max\limits_{1\leqslant i\leqslant p}\{\alpha_i(x+1)+b_i\} -\min\limits_{1\leqslant i\leqslant p}\{\alpha_ix+b_i\}} .
\end{gather}
Assume that the sequence $\big(F^n(u)\big)_n$ starting at $u=\min(H_{\delta}(K_0))$ converges to a number $c\in{\mathbb R}$. Then,
\begin{enumerate}\itemsep=0pt
\item[$(i)$] $(d_n)_n$ converges to $d=\max\limits_{1\leqslant i\leqslant p}\{\alpha_i(c+1)+b_i\}- \min\limits_{1\leqslant i\leqslant p}\{\alpha_i c+b_i\}$,
\item[$(ii)$] If $d> \max\limits_{1\leqslant i \leqslant p} \{\alpha_i\}$ then $\big(H_{\delta}^n(K_0)\big)_n$ converges to $L_d$ whose convex hull is $[c,c+1]$.
\end{enumerate}
\end{Proposition}

\begin{proof} Let us write $\ch(K_n)=[a_n,a_n +1]$ for any $n\geqslant1$.

(i) Let $d(x)=\max\limits_{1\leqslant i\leqslant p}\{\alpha_i (x+1)+b_i\}-\min\limits_{1\leqslant i\leqslant p}\{\alpha_i x+b_i\}$. Since $\alpha_i\geqslant 0$ one checks that $d_n=d(a_n)$ and $a_{n+1}=F(a_n)$ where $F$ is given by \eqref{eq:F}. The result follows.

(ii) Since $\lambda_H=\max\limits_{1\leqslant i\leqslant p}\{\alpha_i\}$ the result is obtained by applying Theorem~\ref{theo:renorm}. Finally, since $\ch(L_d)=\lim\limits_{n\to\infty}\ch(K_n)$, we get the last part of the assertion.
\end{proof}

This Proposition gives a very simple and practical tool to prove the convergence of~$(K_n)_n$. It is enough to study the orbits of $F$ which is just a piecewise homographic function.

\begin{Example}\label{exa:limcantorD1} Let us consider the IFS $\{f_1,f_2,f_3\}$ where the $f_i\colon {\mathbb R}\to {\mathbb R}$ are given by
\begin{gather*}
f_1(x)=2x+1 ,\qquad f_2(x)=3x-4 \qquad \text{and}\qquad f_3(x)=x+2 .
\end{gather*}
Then, for all $K_0\in {\cal K}$, the sequence $(K_n)_n$ converges to the attractor $L_d$ with $d=5+2\sqrt{3}$. In particular its convex hull is the interval $\big[1-\sqrt{3},2-\sqrt{3}\big]$.
\end{Example}

\begin{proof} First we determine the function $F$. We obtain
\begin{gather*}
F(x) = \left(\frac{3x-4}{7-2x}\right)\indic_{(-\infty,0]}(x) + \left(\frac{3x-4}{7-x}\right)\indic_{(0,3]}(x)\\
\hphantom{F(x) =}{} + \left(\frac{x+2}{x+1}\right)\indic_{(3,4]}(x) + \left(\frac{x+2}{2x-3}\right)\indic_{(4,+\infty)}(x).
\end{gather*}
The only invariant point of $F$ is $c=1-\sqrt{3}$ and one checks that $(F^n(u))_n$ converges to $c$ for every $u\in {\mathbb R}$. Thus $(d_n)_n$ converges to $d=f_3(c+1)-f_2(c)=5+2\sqrt{3}$. Since $d>\max\{\alpha_1,\alpha_2,\alpha_3\}=5$, the result follows from Proposition~\ref{prop:diamD1}.
\end{proof}

\begin{Example}\label{exa:diamtrans} Let $p\geqslant1$ affine maps $f_i\colon {\mathbb R}\rightarrow {\mathbb R}$ of the form $f_i(x)=\alpha x + b_i$ with $\alpha\geqslant0$ and $b_1<\cdots <b_p$. Then, for all $K_0\in {\cal K}$, the sequence $(K_n)_n$ converges to $L_d$ with $d=\alpha+(b_p-b_1)$.
\end{Example}

\begin{proof} We get
\begin{gather*}
F(x) = \frac{\alpha x +b_1}{\alpha+(b_p-b_1)}.
\end{gather*}
Thus $F$ is a contraction, $(a_n)_{n\geqslant1}$ converges to the invariant point $c=\frac{b_1}{b_p-b_1}$ and $(d_n)_{n\geqslant1}$ is constant to $d= \alpha+(b_p-b_1)$. Since $d>\alpha=\lambda_H$, the result follows from Proposition \ref{prop:diamD1}.
\end{proof}

\subsection*{Acknowledgements}

The present paper has been completed during the thematic research semester \textit{Fractal Geometry and Dynamics} organized in the fall of 2017 at the Institut Mittag-Leffler, Stockholm, Sweden. The author is very grateful to the organizers for their warm welcome during its stay at the Institute. This work is partially supported by the French research group `Analyse Multifractale' (CNRS-GDR3475).

\pdfbookmark[1]{References}{ref}
\LastPageEnding

\end{document}